\documentclass[12pt, twoside]{article}
\usepackage{amsmath,amsthm,amssymb}
\usepackage{times}
\usepackage{enumerate}
\usepackage{url}
\usepackage{microtype}
\usepackage[american]{babel}
\usepackage{indentfirst}
\usepackage{color}
\usepackage{geometry}
\usepackage{booktabs}
\usepackage{epsfig}
\usepackage[pagewise]{lineno}
\def\titlerunning#1{\gdef\titrun{#1}}
\makeatletter
\def\author#1{\gdef\autrun{\def\and{\unskip, }#1}\gdef\@author{#1}}
\def\address#1{{\def\and{\\\hspace*{18pt}}\renewcommand{\thefootnote}{}%
\footnote {#1}}%
\markboth{\autrun}{\titrun}}
\makeatother
\def\email#1{e-mail: #1}
\def\subjclass#1{{\renewcommand{\thefootnote}{}%
\footnote{\emph{Mathematics Subject Classification (2020):} #1}}}
\def\keywords#1{\par\medskip
\noindent\textbf{Keywords.} #1}

\newtheorem{result}{\textbf{Theorem}}

\newtheorem{proposition}{Proposition}[section]
\newtheorem{lemma}{Lemma}[section]
\newtheorem{definition}{Definition}[section]
\newtheorem{corollary}{Corollary}[section]
\newtheorem{remark}{Remark}[section]
\newtheorem{example}{Example}[section]
\numberwithin{equation}{section}

\begin{document}
\baselineskip=15pt

\titlerunning{Weakly coupled Hamilton-Jacobi systems}
\title{Weakly coupled Hamilton-Jacobi systems without monotonicity condition: A first step}
\author{Panrui Ni
}

\maketitle
\address{Shanghai Center for Mathematical Sciences,  Fudan University, Shanghai 200438, China;
\email{prni18@fudan.edu.cn}
}


\vspace{-10ex}


\begin{abstract}
In this paper, we mainly focus on the existence of the viscosity solutions of
\begin{equation*}
  \left\{
   \begin{aligned}
   &H_1(x,Du_1(x),u_1(x),u_2(x))=0,\\
   &H_2(x,Du_2(x),u_2(x),u_1(x))=0.\\
   \end{aligned}
   \right.
\end{equation*}
The standard assumption for the above system is called the monotonicity condition, which requires that $H_i$ is increasing in $u_i$ and decreasing in $u_j$ for each $i,j\in\{1,2\}$ and $i\neq j$. In this paper, it is assumed that $H_i$ is either increasing or decreasing in $u_i$, and may be non-monotone in $u_j$. The existence of viscosity solutions is proved when
\[\chi:=\sup_{u,v,w\in\mathbb R}\bigg|\frac{\partial_{u_2} H_1(x,0,0,u)}{\partial_{u_1} H_1(x,0,v,w)}\bigg|\cdot
\sup_{u,v,w\in\mathbb R}\bigg|\frac{\partial_{u_1} H_2(x,0,0,u)}{\partial_{u_2} H_2(x,0,v,w)}\bigg|<1.\]
Then we consider
\begin{equation*}
  \left\{
   \begin{aligned}
   &h_1(x,Du_1(x))+\Lambda_1(x)(u_1(x)-u_2(x))=c,\\
   &h_2(x,Du_2(x))+\Lambda_2(x)(u_2(x)-u_1(x))=\alpha(c).\\
   \end{aligned}
   \right.
\end{equation*}
It turns out that for each $c\in\mathbb R$, there is a unique constant $\alpha(c)\in\mathbb R$ such that the above system has viscosity solutions. The function $c\mapsto \alpha(c)$ is non-increasing and Lipschitz continuous. In the appendix, the large time convergence of the viscosity solution of evolutionary weakly coupled systems is proved when $\chi<1$.

\end{abstract}




\section{Introduction and main results}

The present paper focuses on weakly coupled systems of Hamilton-Jacobi equations. In this paper, we assume that $M$ is a connected, closed (compact without boundary) and smooth Riemannian manifold. We denote by $T^*M$ the cotangent bundle over $M$, and denote by $|\cdot|_x$ the norms induced by the Riemannian metric $g$ on both tangent and cotangent spaces of $M$. We also denote by $D$ the spatial gradient with respect to $x\in M$, and denote by $C(M)$ (resp. $C(M,\mathbb R^2)$) the space of $\mathbb R$-valued (resp. $\mathbb R^2$-valued) continuous functions on $M$. For $i\in\{1,2\}$, let $H_i:T^*M\times\mathbb R^2\to\mathbb R$ be a continuous function. We will deal with the viscosity solutions of
\begin{equation}\label{E}
  H_i(x,Du_i(x),u_i(x),u_j(x))=0,\quad x\in M,\ i,j\in \{1,2\},\ i\neq j
\end{equation}
in this paper and thus we mean by ``solutions" viscosity solutions. The unknown function in (\ref{E}) is $(u_1,u_2)\in C(M,\mathbb R^2)$. The system is weakly coupled in the sense that every $i$th equation depends only on $D u_i$, but not on $Du_j$ for $j\neq i$. Studies on the system (\ref{E}) and other related ones are motivated by searching for optimal strategy of classical or random switching problems, see \cite{opt1, opt2} and the references therein.

\begin{definition}(Viscosity solution).
An upper semi-continuous (u.s.c. in short) (resp. lower semi-continuous (l.s.c. in short)) function $(u_1,u_2):M\rightarrow\mathbb R^2$ is called a viscosity subsolution (resp. supersolution) of (\ref{E}), if for each $i,j\in\{1,2\}$, $i\neq j$ and test function $\phi$ of class $C^1$, when $u_i-\phi$ attains its local maximum (resp. minimum) at $x$, then
\begin{equation*}
  H_i(x,D\phi(x),u_i(x),u_j(x))\leq 0,\quad (\textrm{resp}.\ H_i(x,D\phi(x),u_i(x),u_j(x))\geq 0).
\end{equation*}
A function $(u_1,u_2):M\to\mathbb R^2$ is called a viscosity solution of (\ref{E}) if its u.s.c. and l.s.c. envelopes are respectively a viscosity subsolution and a viscosity supersolution. In this paper, we only consider continuous viscosity solutions.
\end{definition}
\noindent The standard assumption for (\ref{E}) is the classical monotonicity condition:
\begin{itemize}
\item [$\diamond$] for any $(x,p)\in T^*M$ and $(u_1,u_2)$, $(v_1,v_2)\in\mathbb R^2$, if $u_k-v_k=\max_{1\leq i\leq 2}(u_i-v_i)\geq 0$, then $H_k(x,p,u_1,u_2)\geq H_k(x,p,v_1,v_2)$.
\end{itemize}
This condition implies that
\begin{equation}\label{inde}
  H_i(x,p,u_i,u_j)\ \textrm{is}\ \textrm{increasing}\ \textrm{in}\ u_i\ \textrm{and}\ \textrm{nonincreasing}\ \textrm{in}\ u_j\ \textrm{for\ each}\ i, j \in \{1,2\},\ i\neq j.
\end{equation}
When the coupling is linear, that is, when $H_i$ has the form
\begin{equation}\label{LH}
  H_i(x,p,u_1,u_2)=h_i(x,p)+\sum_{j=1}^2\lambda_{ij}(x)u_j,
\end{equation}
the classical monotonicity condition ($\diamond$) holds if and only if
\begin{equation}\label{stan}
  \lambda_{ij}(x)\leq 0\ \textrm{if}\ i\neq j\quad \textrm{and}\quad \sum_{j=1}^2\lambda_{ij}(x)\geq 0\ \textrm{for\ all}\ i\in\{1,2\}.
\end{equation}
If the coupling matrix $(\lambda_{ij}(x))$ is irreducible, by (\ref{stan}) we have
\begin{equation}\label{mono2}
  \lambda_{11}(x)>0,\quad \lambda_{22}(x)>0,\quad \lambda_{12}(x)<0, \quad\lambda_{21}(x)<0,
\end{equation}
and
\begin{equation}\label{mcl}
  \lambda_{11}(x)+\lambda_{12}(x)\geq 0, \quad \lambda_{22}(x)+\lambda_{21}(x)\geq 0.
\end{equation}
In the present paper, an existence result of solutions of (\ref{E}) will be provided, where (\ref{inde}) may not hold. From a theoretical point of view, this result can describe a more general form of coupling. From the perspective of game theory, this can characterize influence between players. The value function $u_i(x)$ of the $i$th player can depend either increasingly or decreasingly on the value function $u_j(x)$ of the $j$th player. The existence of solutions of (1.1) can be interpreted as the existence of equilibrium points.

About the weakly coupled systems, there are several topics of concern:
\begin{itemize}
\item [$\centerdot$] The stationary weakly coupled systems. For the existence theorems and the comparison results, one can refer to \cite{CGT,Eng,Len}, and \cite{ABR,Ish3} for the second order case. An algorithm was constructed in \cite{SZ} which allows obtaining a solution as the limit of a monotonic sequence of subsolutions. For the weak KAM theory, one can refer to \cite{Davi1}. For the vanishing discount problem, one can refer to \cite{Davi2,Ish5}.

\item [$\centerdot$] The evolutionary weakly coupled systems. 
The representation formula is provided in the deterministic setting by \cite{JL2} and in the random setting by \cite{ran}. One can refer to \cite{Z} for the representation in view of the twisted Lax-Oleinik formula. For the large time behavior, one can refer to \cite{vis1,Cam2,Mit2,Ng}. For the homogenization theory, one can refer to \cite{Cam1,Mit4}.
\end{itemize}
All the previous works mentioned above require the classical monotonicity condition ($\diamond$) except \cite{JL2}. Particularly, most works focus on the linear coupling case with $k$ coupled Hamilton-Jacobi equations ($k\geq 2$), that is,
\begin{equation}\label{LEk}
  h_i(x,Du_i(x))+\sum_{j=1}^k \lambda_{ij}(x)u_j(x)=0,\quad x\in M,\ i\in\{1,\dots,k\},
\end{equation}
where for each $i,j\in\{1,\dots,k\}$, $i\neq j$, we have
\begin{equation}\label{sum=0}
  \lambda_{ii}(x)>0,\quad \lambda_{ij}(x)\leq 0,\quad \sum_{l=1}^k\lambda_{il}(x)=0.
\end{equation}
For the case $k=2$ considered in the present paper, (\ref{LEk}) becomes
\begin{equation}\label{LE}
  h_i(x,Du_i(x))+\sum_{j=1}^2\lambda_{ij}(x)u_j(x)=0,\quad x\in M,\ i\in\{1,2\}.
\end{equation}
When (\ref{sum=0}) holds, (\ref{LE}) becomes
\begin{equation}\label{=c}
  h_i(x,Du_i(x))+\Lambda_{i}(x)(u_i(x)-u_j(x))=c,\quad x\in M,\ i,j\in\{1,2\},\ i\neq j.
\end{equation}
In general, (\ref{=c}) does not have solutions if $c=0$. It was proved in \cite[Theorem 2.12]{Davi1} that there is a unique constant $c$ such that (\ref{=c}) has solutions. This result is quite similar to that for single equations. Since (\ref{=c}) is a system, it is natural to consider the situation
\[h_i(x,Du_i(x))+\Lambda_{i}(x)(u_i(x)-u_j(x))=c_i,\]
where $c_1\neq c_2$. This situation has been considered in \cite[Theorem 1.3]{CGT} when $\Lambda_1(x)$ and $\Lambda_2(x)$ are constants. A generalized form of this result will be given in this paper.

\subsection{Nonlinear coupling}\label{1.1}

The main assumptions of the Hamiltonians $H_i:T^*M\times\mathbb R^2\rightarrow\mathbb R$ with $i,j\in \{1,2\}$ and $i\neq j$ are given below.
\begin{itemize}
\item [\textbf{(H1)}] $H_i(x,p,u_i,u_j)$ is continuous.

\item [\textbf{(H2)}] $H_i(x,p,u_i,u_j)$ is superlinear in $p$, i.e. there exists a function $\theta:\mathbb R\rightarrow\mathbb R$ satisfying
\begin{equation*}
  \lim_{r \rightarrow+\infty}\frac{\theta(r)}{r}=+\infty,\quad \textrm{and}\quad H_i(x,p,0,0)\geq \theta(|p|_x)\quad \textrm{for\ every}\ (x,p)\in T^*M.
\end{equation*}


\item [\textbf{(H3)}] $H_i(x,p,u_i,u_j)$ is convex in $p$.

\item [\textbf{(H4)}] $H_i(x,p,u_i,u_j)$ is uniformly Lipschitz continuous in $u_i$ and $u_j$, i.e., there is $\Theta>0$ such that
\[|H_i(x,p,u_i,u_j)-H_i(x,p,v_i,v_j)|\leq \Theta\max\{|u_i-v_i|,|u_j-v_j|\}.\]

\item [\textbf{(H5)}] $H_i(x,p,u_i,u_j)$ is strictly increasing in $u_i$, there exists $\lambda_{ii}>0$ such that for all $(x,p,v)\in\ T^*M\times\mathbb R$ we have
\begin{equation*}
   H_i(x,p,u_i,v)-H_i(x,p,v_i,v)\geq \lambda_{ii}(u_i-v_i),\quad \forall u_i\geq v_i.
\end{equation*}

\item [\textbf{(H6)}] $H_i(x,p,u_i,u_j)$ is strictly decreasing in $u_i$, there exists $\lambda_{ii}>0$ such that for all $(x,p,v)\in\ T^*M\times\mathbb R$ we have
\begin{equation*}
   H_i(x,p,u_i,v)-H_i(x,p,v_i,v)\leq -\lambda_{ii}(u_i-v_i),\quad \forall u_i\geq v_i.
\end{equation*}
\item [\textbf{(H7)}] there are two constants $b_{12}>0$ and $b_{21}>0$ such that for all $i, j\in\{1,2\}$, $i\neq j$, for all $(x,v)\in M\times\mathbb R$ and for all $u_i\neq v_i,u_j\neq v_j$,
    \begin{equation*}
    \bigg|\frac{H_i(x,0,0,u_j)-H_i(x,0,0,v_j)}{u_j-v_j}\cdot \frac{u_i-v_i}{H_i(x,0,u_i,v)-H_i(x,0,v_i,v)}\bigg|
   \leq b_{ij}.
\end{equation*}
\end{itemize}
In the proof of Theorem \ref{M1}, as an intermediate step, we will replace (H2) by
\begin{itemize}
\item [\textbf{(H8)}] $H_i(x,p,u_i,u_j)$ is coercive in $p$, i.e. $\lim_{|p|_x\rightarrow +\infty}(\inf_{x\in M}H_i(x,p,0,0))=+\infty$.
\end{itemize}
\begin{remark}
Let $H_i(x,p,u_i,u_j)$ satisfy either (H5) or (H6). If there is a constant $\lambda_{ij}>0$ such that
\[|H_i(x,0,0,u_j)-H_i(x,0,0,v_j)|\leq \lambda_{ij}|u_j-v_j|,\quad \forall x\in M,\ \forall u_j,v_j\in\mathbb R,\]
then (H7) holds if we take $b_{ij}=\lambda_{ij}/\lambda_{ii}$. 
When the Hamiltonian $H_i$ is smooth, the condition (H7) can be guaranteed by
\[\bigg|\frac{\partial_{u_j}H_i(x,0,0,u)}{\partial_{u_i}H_i(x,0,v,w)}\bigg|\leq b_{ij},\quad \forall x\in M,\ \forall u,v,w\in\mathbb R,\ \forall i, j\in\{1,2\},\ i\neq j.\]
For the linear coupling case (\ref{LH}) with a continuous coupling matrix satisfying $\lambda_{11}(x)\neq 0$ and $\lambda_{22}(x)\neq 0$, (H7) automatically holds, and the constants $b_{ij}$ in (H7) are given by
\begin{equation*}\label{bij}
  b_{ij}=\max_{x\in M}\left|\frac{\lambda_{ij}(x)}{\lambda_{ii}(x)}\right|.
\end{equation*}
\end{remark}

Now we define an important constant \[\chi:=b_{12}b_{21}.\] When $\chi$ is small, the coupling is thought to be weak.

Throughout this paper, we call \textbf{(I)} the conditions (H1)(H2)(H4)(H5)(H7), and \textbf{(D)} the conditions (H1)(H2)(H3)(H4)(H6)(H7). It is worth mentioning that, when $H_i(x,p,u_i,u_j)$ is increasing in $u_i$, the convexity assumption (H3) is not needed, see Remark \ref{rem3} below.

\begin{result}\label{M1}
The system (\ref{E}) has viscosity solutions if each $H_i(x,p,u_i,u_j)$ satisfies either the condition (I) or the condition (D), and $\chi<1$.
\end{result}

\begin{remark}\label{detB}
To see what is new in the above theorem, let us consider the linear coupling case (\ref{LE}). To apply Perron's method \cite[Theorem 3.3]{Ish3}, (\ref{mcl}) is not enough, we need (\ref{mono2}) and
\begin{equation}\label{perron}
  \lambda_{11}(x)+\lambda_{12}(x)>0, \quad \lambda_{22}(x)+\lambda_{21}(x)>0.
\end{equation}
Then one can take $C>0$ large enough such that $(C,C)$ (resp. $(-C,-C)$) is a supersolution (resp. subsolution) of (\ref{LE}). 
Comparing to (\ref{mono2}), Theorem \ref{M1} can handle the following cases
\[\textrm{Case\ 1.}\ \lambda_{11}(x)>0, \lambda_{22}(x)>0,\quad \textrm{Case\ 2.}\ \lambda_{11}(x)>0, \lambda_{22}(x)<0,\]
\[\textrm{Case\ 3.}\ \lambda_{11}(x)<0, \lambda_{22}(x)>0,\quad \textrm{Case\ 4.}\ \lambda_{11}(x)<0, \lambda_{22}(x)<0,\]
and $\lambda_{12}(x),\lambda_{21}(x)$ are allowed to change signs. Let (\ref{mono2}) holds. It is obvious that (\ref{mcl}) implies $\chi\leq 1$, and (\ref{perron}) implies $\chi<1$. In Proposition \ref{perronequ} below, it will be shown that Theorem \ref{M1} has nothing new compared to (\ref{perron}) when (\ref{mono2}) holds.
Although $\chi<1$ and (\ref{perron}) are essentially equivalent, the condition $\chi<1$ can include more situations than (\ref{perron}), see (\ref{Ee}) below as an example, where $\lambda_{11}(x)+\lambda_{12}(x)=0$ and $\lambda_{22}(x)+\lambda_{21}(x)>0$. Here is another simple example that cannot be included in (\ref{perron}): $\lambda_{11}=1$, $\lambda_{12}=-2$, $\lambda_{22}=4$ and $\lambda_{21}=-1$. Here $\lambda_{11}+\lambda_{12}<0$ and $\chi<1$. Let $\Lambda(x)=(\lambda_{ij}(x))_{i,j\in\{1,2\}}$ be the coupling matrix. When $\Lambda(x)$ is a constant matrix, and (\ref{mono2}) holds, the condition $\chi<1$ is equivalent to
\[\det \Lambda=\lambda_{11}\lambda_{22}-\lambda_{12}\lambda_{21}>0,\]
which is a better condition depending on the whole weakly coupled system comparing to (\ref{perron}). When (\ref{mono2}) fails, the existence result in Theorem \ref{M1} is new. 
\end{remark}

\begin{remark}
By Lemma \ref{stuni}, the solution of (\ref{LE}) is unique when $\chi<1$ and (\ref{mono2}) hold. When (\ref{mono2}) fails, the uniqueness may not hold. Here is an example
\begin{equation*}
  \left\{
   \begin{aligned}
   &|Du_1|^2+2u_1-u_2=0,\\
   &|Du_2|^2-4u_2-u_1=0,\\
   \end{aligned}
   \right.
\end{equation*}
where $x$ belongs to the unit circle $\mathbb S^1\simeq [-1,1)$. Then $\chi=1/8<1$. The first equation is increasing in $u_1$, while the second equation is decreasing in $u_2$. Let $f(x)$ be the restriction of $x^2$ on $[-1,1)$. There are two solutions of the above system: $u^{(1)}_1=u^{(1)}_2=0$, and
\[u^{(2)}_1=\frac{1}{8}(\sqrt{13}-1)f(x), \quad u^{(2)}_2=\frac{1}{8}(\sqrt{13}+5)f(x).\]
\end{remark}

\begin{remark}\label{rem1}
The condition (H2) can be relaxed to (H8) when $H_i$ satisfies (H5) for all $i\in\{1,2\}$. See Proposition \ref{proI-I} in Section \ref{a}. If there is $i\in\{1,2\}$ such that $H_i$ satisfies (H6) and (H8), $\chi<1$ must be replaced by a more complicated one, see Propositions \ref{proD-D} and \ref{proI-D} below.
\end{remark}

\begin{remark}\label{rem3}
Consider the single Hamilton-Jacobi equation
\begin{equation}\label{hjj}
  H(x,D u(x),u(x))=0,\quad x\in M.
\end{equation}

\begin{itemize}
\item [(a)] $H(x,p,u)-H(x,p,v)\geq \delta (u-v)$ for all $u>v$ and for some $\delta>0$.

\item [(b)] $H(x,p,u)-H(x,p,v)\leq -\delta (u-v)$ for all $u>v$ and for some $\delta>0$.
\end{itemize}



\noindent For Case (a), the existence of solution of (\ref{hjj}) is given by Perron's method. This is why (H3) is not needed when $H_i$ is increasing in $u_i$. For Case (b), the existence of solutions of (\ref{hjj}) is given by Proposition \ref{m3}. The proof relies on the solution semigroup, so the convexity of $H$ in $p$ is needed. This explains why we need (H3) when $H_i$ is decreasing in $u_i$. Different from Proposition \ref{m3}, we need the Hamiltonian to be superlinear in $p$ to get the existence of solutions of (\ref{E}) when there is $i\in\{1,2\}$ such that (H6) holds. 
\end{remark}

\subsection{Linear coupling with monotonicity}\label{1.2}

For $i,j\in\{1,2\}$, we assume that $h_i:T^*M\to\mathbb R$ and $\lambda_{ij}(x)$ in (\ref{LE}) are continuous. Assume




\begin{itemize}
\item [\textbf{(h1)}] $h_i(x,p)$ is coercive in $p$, i.e. $\lim_{|p|_x\rightarrow +\infty}(\inf_{x\in M}h_i(x,p))=+\infty$.


\end{itemize}

Consider the following Cauchy problem
\begin{equation}\label{LCau}
  \left\{
   \begin{aligned}
   &\partial_t u_i(x,t)+h_i(x,Du_i(x,t))+\sum_{j=1}^2\lambda_{ij}(x)u_j(x,t)=0,\quad i\in\{1,2\},\\
   &u_i(x,0)=\varphi_i(x)\in C(M).\\
   \end{aligned}
   \right.
\end{equation}
\begin{result}\label{M4}
Assume (h1)(\ref{mono2}). When $\chi<1$, the solution $(v_1,v_2)$ of (\ref{LE}) exists and is unique, and the unique solution $(u_1(x,t),u_2(x,t))$ of (\ref{LCau}) uniformly converges to $(v_1,v_2)$ as $t\to+\infty$ for each continuous initial function $(\varphi_1,\varphi_2)$.
\end{result}

\begin{remark}
The proof of Theorem \ref{M4} is a standard PDE argument, and is based on the comparison principle, we provide it in Appendix \ref{sec3}. When the coupling matrix $\Lambda(x)=(\lambda_{ij}(x))_{i,j\in\{1,2\}}$ is independent of $x\in M$, Theorem \ref{M4} implies the large time convergence when (h1)(\ref{mono2}) and \[\det\Lambda>0\] hold. The large time behavior of solutions corresponding to (\ref{=c}) has been considered in \cite{vis1,Cam2,Mit2,Ng}, where $\chi=1$. Consider the solution $u(x,t)$ of the single Hamilton-Jacobi equation
\begin{equation}\label{ssC}
  \left\{
   \begin{aligned}
   &\partial_t u(x,t)+H(x,D u(x,t),u(x,t))=0,\quad (x,t)\in M\times(0,+\infty).
   \\
   &u(x,0)=\varphi(x),\quad x\in M.
   \\
   \end{aligned}
   \right.
\end{equation}
When $H(x,p,u)$ is strictly increasing in $u$, then the solution $v(x)$ of (\ref{hjj}) is unique by \cite[Theorem 3.2]{inc}. By \cite{Su}, $u(x,t)$ uniformly converges to $v(x)$ as $t\to+\infty$ for all initial function $\varphi$. Theorem \ref{M4} generalizes these results to weakly coupled systems.
\end{remark}

Now we assume

\begin{itemize}
\item [\textbf{(h2)}] $h_i(x,p)$ is locally Lipschitz continuous.

\item [\textbf{(h3)}] $h_i(x,p)$ is strictly convex in $p$.
\end{itemize}
Assumptions (h2) and (h3) guarantee the semiconcavity of viscosity solutions. In the following, we assume that $\Lambda_1(x)$ and $\Lambda_2(x)$ are two positive functions on $M$, and are Lipschitz continuous.






The following result generalizes \cite[Theorem 2.12]{Davi1} for weakly coupled systems with two Hamilton-Jacobi equations. It also generalizes \cite[Theorem 1.3]{CGT} in the case where $\Lambda_i(x)$ depends on $x$. In this case, $\chi=1$.
\begin{result}\label{M3}
Assume (h1)(h2)(h3). There is a unique constant $\alpha(c)\in\mathbb R$ such that
\begin{equation}\label{CE}
  \left\{
   \begin{aligned}
   &h_1(x,Du_1(x))+\Lambda_1(x)(u_1(x)-u_2(x))=c,\\
   &h_2(x,Du_2(x))+\Lambda_2(x)(u_2(x)-u_1(x))=\alpha(c).\\
   \end{aligned}
   \right.
\end{equation}
admits viscosity solutions. The function $c\mapsto \alpha(c)$ is nonincreasing with the Lipschitz constant $\max_{x\in M}\Lambda_2(x)/\min_{x\in M}\Lambda_1(x)$.
\end{result}

The following result has been given by \cite[Theorem 1.3]{CGT}. We will prove it in Section \ref{Sec4} as a corollary of Theorem \ref{M3}.
\begin{corollary}\label{cons}
If $\Lambda_1(x)$ and $\Lambda_2(x)$ are constant functions, then $\alpha(c)=\alpha(0)-(\Lambda_2/\Lambda_1)c$.
\end{corollary}

By the continuity, there is a unique constant $c_0\in\mathbb R$ such that $\alpha(c_0)=c_0$. We have the following result, which is covered by \cite[Theorem 2.12]{Davi1}.
\begin{corollary}\label{H=c}
There is a unique constant $c_0\in \mathbb R$ such that
\[h_i(x,Du_i(x))+\Lambda_i(x)(u_i(x)-u_j(x))=c_0,\quad x\in M,\ i, j\in\{1,2\},\ i\neq j\]
admits viscosity solutions.
\end{corollary}


\subsection{Examples without monotonicity}

If the assumption (\ref{mono2}) does not hold, the existence of solutions of (\ref{LE}) with $\chi=1$ and the large time behavior of solutions of (\ref{LCau}) with $\chi<1$ remain unsolvable. In this section, several examples are provided showing the complexity of the non-monotone cases.
\begin{example}
Consider
\begin{equation}\label{exx}
  \left\{
   \begin{aligned}
   &h(x,Du_1(x))+u_1(x)+u_2(x)=c_1,\\
   &h(x,Du_2(x))+u_2(x)-u_1(x)=c_2.\\
   \end{aligned}
   \right.
\end{equation}
Then $\chi=1$ and (\ref{mono2}) fails. Let $u_0$ be the unique solution of $h(x,Du)+2u=0$. Then for each $(c_1,c_2)\in\mathbb R^2$, the pair $(u_0+\frac{c_1-c_2}{2},u_0+\frac{c_1+c_2}{2})$ is a solution of (\ref{exx}). Therefore, one cannot expect the uniqueness of $c_2$ for given $c_1$ as in Theorem \ref{M3} when $\chi=1$ and (\ref{mono2}) fails.
\end{example}

\begin{example}\label{ex1.2}
For single Hamilton-Jacobi equations, when $H(x,p,u)$ is strictly decreasing in $u$, the large time behavior of solutions can be complicated, see \cite{Wa5}. Here we give an example of time periodic solutions of weakly coupled systems where $\chi<1$ and (\ref{inde}) does not hold. Let $\mathbb S^1$ be the unit circle, and $(x,t)\in \mathbb S^1\times(0,+\infty)$. Let $k>0$, consider
\begin{equation}\label{peri}
\left\{
   \begin{aligned}
   &\partial_t u_1+H(x,Du_1,2u_1-u_2/k)=0,\\
   &\partial_t u_2+kH(x,Du_2/k,2u_2/k-u_1)=0.\\
   \end{aligned}
   \right.
\end{equation}
where $H(x,p,u)$ satisfies the assumptions in \cite{Wa5}, then $-\Theta\leq \partial H/\partial u\leq -\delta<0$. We have
\begin{align*}
&\lambda_{11}=2\delta,\quad \lambda_{12}=\Theta/k,\quad \lambda_{22}=2\delta/k,\quad \lambda_{21}=\Theta,\quad \chi=\frac{\lambda_{12}\lambda_{21}}{\lambda_{11}\lambda_{22}}=\frac{\Theta^2}{4\delta^2}.
\end{align*}
We take $\delta\leq \Theta<2\delta$, then $\chi<1$. According to \cite{Wa5}, there are infinitely many time periodic viscosity solutions of
\begin{equation}\label{per2}
  \partial_tu(x,t)+H(x,Du(x,t),u(x,t))=0,\quad (x,t)\in M\times(0,+\infty).
\end{equation}
Let $\varphi(x,t)$ be a non-trivial time periodic solution of (\ref{per2}), then $(u_1,u_2)=(\varphi,k\varphi)$ is a non-trivial time periodic solution of (\ref{peri}). Therefore, one cannot expect the large time convergence as in Theorem \ref{M4} when $\chi<1$ and (\ref{inde}) fails.
\end{example}

\begin{example}
When $\chi\in (1,+\infty]$, the large time behavior can be even more complicated. 
Let $\mathbb S^1$ be the unit circle, and $(x,t)\in \mathbb S^1\times(0,+\infty)$. Consider
\begin{equation*}
  \left\{
   \begin{aligned}
   &\partial_t u_1+|D u_1|^2+u^2_1-u_2-1=0,
   \\
   &\partial_t u_2+|D u_2|^2+u^2_2+u_1-1=0.
   \\
   \end{aligned}
   \right.
\end{equation*}
For the system above, $|\partial_{u_j}H_i/\partial_{u_i}H_i|=1/|2u_i|\to +\infty$ as $u_i\to 0$. The above system has the following time-periodic solution
\begin{equation*}
  \left\{
   \begin{aligned}
   &u_1=\sin(x+t),
   \\
   &u_2=\cos(x+t).
   \\
   \end{aligned}
   \right.
\end{equation*}
The two components of the non-trivial time periodic solution given in Example \ref{ex1.2} are essentially the same. In this example, the two components of the solution are essentially different. In addition, the author believes that the time periodic solutions of the weakly coupled Hamilton-Jacobi systems can describe the dynamic equilibrium of differential games with multiple types of players. For related topics, see for example \cite{Ni2}.
\end{example}




This paper is organized as follows. Theorem \ref{M1} is proved in Section \ref{Sec3}. Theorem \ref{M4} is proved in Appendix \ref{sec3}. Theorem \ref{M3} and Corollary \ref{cons} are proved in Section \ref{Sec4}. Appendix \ref{Sec2} provides some facts about viscosity solutions of single Hamilton-Jacobi equations depending Lipschitz continuously on the unknown function. These results are useful in the proof of Theorem \ref{M1}. 


\section{Proof of Theorem \ref{M1}}\label{Sec3}

In this section, we divide the proof of Theorem \ref{M1} into three different cases. We first relax the superlinearity condition (H2) to the coercivity condition (H8). Then we complete the proof of theorem \ref{M1} under (H2). In the following, we call \textbf{(I')} the conditions (H1)(H4)(H5)(H7)(H8), and \textbf{(D')} the conditions (H1)(H3)(H4)(H6)(H7)(H8). Since either (H5) or (H6) is assumed to be hold, the classical monotonicity condition ($\diamond$) can be thought to be replaced by a distinct type of monotonicity.

The strategy of the proof is as follows.
\begin{itemize}
\item Considering the single equation $H_1(x,Du,u,0)=0$, we can obtain a solution $u^0_1$.
\item Considering the single equation $H_2(x,Du,u,u^0_1(x))=0$, we can obtain a solution $u^1_2$.
\item Considering the single equation $H_1(x,Du,u,u^1_2(x))=0$, we can obtain a solution $u^1_1$.
\end{itemize}
Let this process continue. We get the following iteration procedure for $n=0,1,2,\dots$
\begin{equation*}\label{un}
  \left\{
   \begin{aligned}
   &H_1(x,Du^n_1(x),u^n_1(x),u^n_2(x))=0,\\
   &H_2(x,Du^{n+1}_2(x),u^{n+1}_2(x),u^n_1(x))=0,\\
   \end{aligned}
   \right.
\end{equation*}
where $u^0_2\equiv 0$. 
We are going to prove that there is a subsequence of $(u^n_1,u^n_2)$ converges uniformly to a pair $(u,v)$. By the stability of viscosity solutions \cite[Theorem 8.1.1]{Fat-b}, the limit $(u,v)$ is a solution of (\ref{E}).


Assume that $H_i:T^*M\times\mathbb R^2\to\mathbb R$ satisfies (H1)(H3)(H4)(H8) for each $i\in\{1,2\}$. The Lagrangian associated to $H_i(x,p,u_i,u_j)$ is defined by
\begin{equation}\label{Li}
  L_i(x,\dot x,u_i,u_j):=\sup_{p\in T^*_xM}\{\langle \dot x,p\rangle_x-H_i(x,p,u_i,u_j)\},
\end{equation}
where $\langle\cdot,\cdot\rangle_x$ represents the canonical pairing between the tangent space and cotangent space. Similar to \cite[Proposition 2.1]{Ish2}, one can prove the local boundedness of $L_i(x,\dot x,0,0)$:
\begin{lemma}\label{mu}
There exist constants $\delta>0$ and $C>0$ independent of $i$ such that for each $i\in\{1,2\}$, the corresponding Lagrangian $L_i(x,\dot x,0,0)$ satisfies
\[L_i(x,\dot x,0,0)\leq C,\quad \forall (x,\dot x)\in M\times\bar B(0,\delta).\]
Here $\bar B(0,\delta)$ the closed ball given by the norm $|\cdot|_x$ centered at $0$ with radius $\delta$. Define $\mu:=\textrm{diam}(M)/\delta$, where $\textrm{diam}(M)$ is the diameter of $M$.
\end{lemma}

Let $\Theta$, $C$, $\mu$ and $b_{ij}$ be the constants defined in the basic assumptions (H4)(H7) and Lemma \ref{mu}. In this section, we define
\[\mathbb H_i:=\|H_i(x,0,0,0)\|_\infty,\quad A:=\Theta\mu e^{\Theta\mu},\quad B:=C\mu e^{\Theta\mu},\quad \bar b_{ij}:=(1+A)b_{ij}+A.
\]
and
\[
\kappa:=b_{12}b_{21},\quad \bar \kappa:=\bar b_{12}\bar b_{21},\quad \tilde{\kappa}:=b_{12}\bar b_{21}.\]

\subsection{Increasing-increasing case}\label{a}

In this section, we will prove
\begin{proposition}\label{proI-I}
The system (\ref{E}) admits viscosity solutions if $H_i(x,p,u_i,u_j)$ satisfies (I') for each $i\in\{1,2\}$, and $b_{12}b_{21}<1$.
\end{proposition}

\begin{lemma}\label{Iv}
For each $i\in\{1,2\}$ and $v(x)\in C(M)$, there is a unique viscosity solution $u(x)$ of
\begin{equation}\label{u0}
  H_i(x,Du,u,v(x))=0.
\end{equation}
Moreover, we have $\|u(x)\|_\infty\leq \frac{1}{\lambda_{ii}}\mathbb H_i+b_{ij}\|v\|_\infty$.
\end{lemma}
\begin{proof}
By (H7) we have
\begin{equation*}
\begin{aligned}
  &|H_i(x,0,0,v(x))-H_i(x,0,0,0)|\cdot b_{ij}\|v\|_\infty
  \\ &\leq b_{ij}\bigl(H_i(x,0,b_{ij}\|v\|_\infty,v(x))-H_i(x,0,0,v(x))\bigl)|v(x)|
  \\ &\leq \bigl(H_i(x,0,b_{ij}\|v\|_\infty,v(x))-H_i(x,0,0,v(x))\bigl)\cdot b_{ij}\|v\|_\infty,\quad \forall x\in M.
\end{aligned}
\end{equation*}
Therefore, we have
\begin{equation}\label{i1}
  |H_i(x,0,0,v(x))-H_i(x,0,0,0)|\leq H_i(x,0,b_{ij}\|v\|_\infty,v(x))-H_i(x,0,0,v(x)).
\end{equation}
Similarly, we have
\begin{equation}\label{i2}
  |H_i(x,0,0,v(x))-H_i(x,0,0,0)|\leq H_i(x,0,0,v(x))-H_i(x,0,-b_{ij}\|v\|_\infty,v(x)).
\end{equation}
By (H5) and (\ref{i1}) we have
\begin{equation*}\label{H-H}
\begin{aligned}
  &H_i(x,0,\frac{1}{\lambda_{ii}}\mathbb H_i+b_{ij}\|v\|_\infty,v(x))-H_i(x,0,b_{ij}\|v\|_\infty,v(x))
  \\ &+H_i(x,0,b_{ij}\|v\|_\infty,v(x))-H_i(x,0,0,v(x))+H_i(x,0,0,v(x))-H_i(x,0,0,0)\geq \mathbb H_i.
\end{aligned}
\end{equation*}
By (H5) and (\ref{i2}) we have
\begin{equation*}
\begin{aligned}
  &H_i(x,0,0,0)-H_i(x,0,0,v(x))+H_i(x,0,0,v(x))-H_i(x,0,-b_{ij}\|v\|_\infty,v)
  \\ &+H_i(x,0,-b_{ij}\|v\|_\infty,v)-H_i(x,0,-\frac{1}{\lambda_{ii}}\mathbb H_i-b_{ij}\|v\|_\infty,v(x))
  \geq \mathbb H_i,
\end{aligned}
\end{equation*}
Thus for every $x\in M$, we have
\[H_i(x,0,\frac{1}{\lambda_{ii}}\mathbb H_i+b_{ij}\|v\|_\infty,v(x))\geq 0,\]
and \[H_i(x,0,-\frac{1}{\lambda_{ii}}\mathbb H_i-b_{ij}\|v\|_\infty,v(x))\leq 0,\]
which implies that $\frac{1}{\lambda_{ii}}\mathbb H_i+b_{ij}\|v\|_\infty$ (resp. $-\frac{1}{\lambda_{ii}}\mathbb H_i-b_{ij}\|v\|_\infty$) is a supersolution (resp. subsolution) of (\ref{u0}). Therefore, the viscosity solution $u(x)$ of (\ref{u0}) exists by Perron's method \cite{Ish}. The continuity of $u(x)$, 
the uniqueness of $u(x)$ and $\|u(x)\|_\infty\leq \frac{1}{\lambda_{ii}}\mathbb H_i+b_{ij}\|v\|_\infty$ are given by the comparison principle.
\end{proof}
\begin{lemma}\label{unvn}
For $n=1,2,\dots$, we have
\begin{equation*}
  \|u^n_1\|_\infty \leq \frac{\mathbb H_1}{\lambda_{11}}\sum_{l=0}^n\kappa^l+b_{12}\frac{\mathbb H_2}{\lambda_{22}}\sum_{l=0}^{n-1}\kappa^l,
\end{equation*}
and
\begin{equation*}
  \|u^{n+1}_2\|_\infty \leq \frac{\mathbb H_2}{\lambda_{22}}\sum_{l=0}^n\kappa^l+b_{21}\frac{\mathbb H_1}{\lambda_{11}}\sum_{l=0}^{n}\kappa^l.
\end{equation*}
\end{lemma}
\begin{proof}
We prove by induction. We first prove the Lemma when $n=1$. By Lemma \ref{Iv}, we have
\begin{equation*}
\begin{aligned}
  &\|u^0_1(x)\|_\infty\leq \frac{\mathbb H_1}{\lambda_{11}},
  \\ &\|u^1_2(x)\|_\infty\leq \frac{\mathbb H_2}{\lambda_{22}}+b_{21}\|u^0_1(x)\|_\infty \leq \frac{\mathbb H_2}{\lambda_{22}}+b_{21}\frac{\mathbb H_1}{\lambda_{11}},
  \\ &\|u^1_1(x)\|_\infty\leq \frac{\mathbb H_1}{\lambda_{11}}+b_{12}\|u^1_2(x)\|_\infty\leq \frac{\mathbb H_1}{\lambda_{11}}(1+\kappa)+b_{12}\frac{\mathbb H_2}{\lambda_{22}},
  \\ &\|u^2_2(x)\|_\infty\leq \frac{\mathbb H_2}{\lambda_{22}}+b_{21}\|u^1_1(x)\|_\infty\leq \frac{\mathbb H_2}{\lambda_{22}}(1+\kappa)+b_{21}\frac{\mathbb H_1}{\lambda_{11}}(1+\kappa).
\end{aligned}
\end{equation*}

Assume that the assertion holds true for $n=k-1$. We prove the Lemma when $n=k$. By Lemma \ref{Iv}, we have
\begin{equation*}
\begin{aligned}
  \|u^k_1(x)\|_\infty&\leq \frac{\mathbb H_1}{\lambda_{11}}+b_{12}\|u^k_2(x)\|_\infty
  \\ &\leq \frac{\mathbb H_1}{\lambda_{11}}(1+b_{12}b_{21}\sum_{l=0}^{k-1}\kappa^l)+b_{12}\frac{\mathbb H_2}{\lambda_{22}}\sum_{l=0}^{k-1}\kappa^l
  \\ &=\frac{\mathbb H_1}{\lambda_{11}}\sum_{l=0}^k\kappa^l+b_{12}\frac{\mathbb H_2}{\lambda_{22}}\sum_{l=0}^{k-1}\kappa^l,
\end{aligned}
\end{equation*}
and
\begin{equation*}
\begin{aligned}
  \|u^{k+1}_2(x)\|_\infty&\leq \frac{\mathbb H_2}{\lambda_{22}}+b_{21}\|u^k_1(x)\|_\infty
  \\ &\leq \frac{\mathbb H_2}{\lambda_{22}}(1+b_{21}b_{12}\sum_{l=0}^{k-1}\kappa^l)+b_{21}\frac{\mathbb H_1}{\lambda_{11}}\sum_{l=0}^k\kappa^l
  \\ &=\frac{\mathbb H_2}{\lambda_{22}}\sum_{l=0}^k\kappa^l+b_{21}\frac{\mathbb H_1}{\lambda_{11}}\sum_{l=0}^{k}\kappa^l.
\end{aligned}
\end{equation*}
The proof is now completed.
\end{proof}

\begin{lemma}\label{equisub}
Let $h:T^*M\to\mathbb R$ satisfy (h1). Given $c\in\mathbb R$. Then all u.s.c. viscosity subsolutions of
\begin{equation}\label{h=c}
  h(x,Du)=c,\quad x\in M.
\end{equation}
are equi-Lipschitz continuous.
\end{lemma}
\begin{proof}
Since the discussion is local, we assume that $M$ is an open bounded subset of $\mathbb R^n$. Let $w$ be an u.s.c. subsolution of (\ref{h=c}). By definition, for test function $\phi$ of class $C^1$, when $w-\phi$ attains its local maximum at $x$, we have $h(x,D\phi(x))\leq c$. By (h1), there is $\kappa>0$ independent of $x$ such that $\|D\phi(x)\|\leq \kappa$, where $\|\cdot\|$ is a norm in $\mathbb R^n$. Thus, $\|Dw(x)\|\leq \kappa$ holds in the viscosity sense. By \cite[Proposition 1.14]{Ish6}, $w$ is Lipschitz continuous with the Lipschitz constant $\kappa$.
\end{proof}

\noindent {\it Proof of Proposition \ref{proI-I}.} By assumption we have $\kappa<1$. By Lemma \ref{unvn}, both $u^n_1$ and $u^n_2$ are uniformly bounded by a constant $K>0$ independent of $n$. By (H4) we have
\begin{equation*}\label{<K}
  H_i(x,Du^n_i(x),0,0)\leq \Theta K
\end{equation*}
in the viscosity sense. By Lemma \ref{equisub}, $(u^n_1,u^n_2)$ is equi-Lipschitz continuous. By the Arzel\`a-Ascoli theorem, there is a subsequence of $(u^n_1,u^n_2)$ converges uniformly to a pair $(u,v)$. By the stability of viscosity solutions, the limit $(u,v)$ is a solution of (\ref{E}).\qed

\begin{proposition}\label{perronequ}
The existence of solution of (\ref{LE}) can be proved by Perron's method when (\ref{mono2}) and $b_{12}b_{21}<1$ hold. 
\end{proposition}
\begin{proof}
Taking $r>0$, we have
\begin{equation*}
  \left\{
   \begin{aligned}
   &h_1(x,rD(u_1/r))+\lambda_{11}(x)r(u_1/r)+\lambda_{12}(x)u_2=0,\\
   &h_2(x,Du_2)+\lambda_{22}(x)u_2+\lambda_{21}(x)r(u_1/r)=0.\\
   \end{aligned}
   \right.
\end{equation*}
Let $v_1=u_1/r$, then the pair $(v_1,u_2)$ satisfies
\begin{equation*}
  \left\{
   \begin{aligned}
   &h_1(x,rDv_1)+r\lambda_{11}(x)v_1+\lambda_{12}(x)u_2=0,\\
   &h_2(x,Du_2)+\lambda_{22}(x)u_2+r\lambda_{21}(x)v_1=0.\\
   \end{aligned}
   \right.
\end{equation*}
We take $r>b_{12}>0$. Then
\[r\lambda_{11}(x)+\lambda_{12}(x)>b_{12}\lambda_{11}(x)+\lambda_{12}(x)\geq 0.\]
By $\chi<1$, we also have
\[-r\frac{\lambda_{21}(x)}{\lambda_{22}(x)}\leq rb_{21}<1,\]
when $r$ is close to $b_{12}$. Then we get $\lambda_{22}(x)+r\lambda_{21}(x)>0$. Therefore, (\ref{perron}) holds for the pair $(v_1,u_2)$. The existence of $(v_1,u_2)$ implies the existence of $(u_1,u_2)$.
\end{proof}

In the following, we write $(u_1,u_2)\leq (v_1,v_2)$ (resp. $(u_1,u_2)\geq (v_1,v_2)$) if $u_1\leq v_1$ and $u_2\leq v_2$ (resp. $u_1\geq v_1$ and $u_2\geq v_2$). Same for $(u_1,u_2)<(v_1,v_2)$ and $(u_1,u_2)>(v_1,v_2)$.

By \cite[Proposition 2.10]{Davi1}, we have
\begin{lemma}\label{stuni}
Assume (\ref{mono2}) and $b_{12}b_{21}<1$. Let $(\tilde v_1,\tilde v_2)$ (resp. $(\bar v_1,\bar v_2)$) be a continuous subsolution (resp. supersolution) of (\ref{LE}), then $(\tilde v_1,\tilde v_2)\leq (\bar v_1,\bar v_2)$. Moreover, the solution of (\ref{LE}) is unique.
\end{lemma}


\subsection{Decreasing-decreasing case}\label{secd-d}

In this section, we will prove
\begin{proposition}\label{proD-D}
The system (\ref{E}) admits viscosity solutions if $H_i(x,p,u_i,u_j)$ satisfies (D') for each $i\in\{1,2\}$, and there is $\mu>0$ as mentioned in Lemma \ref{mu}, such that
\begin{equation}\label{bark<1}
  \left((1+\Theta\mu e^{\Theta\mu})b_{12}+\Theta\mu e^{\Theta\mu}\right)\left((1+\Theta\mu e^{\Theta\mu})b_{21}+\Theta\mu e^{\Theta\mu}\right)<1.
\end{equation}
\end{proposition}
If $H_i(x,p,u_1,u_2)$ satisfies (D'), the comparison principle does not hold. Therefore, one can not obtain similar results as in Lemma \ref{Iv} directly.

\begin{lemma}\label{dv}
For each $i\in\{1,2\}$ and $v(x)\in C(M)$, the viscosity solutions of (\ref{u0}) exist. For each viscosity solution $u(x)$ of (\ref{u0}), we have \[\|u(x)\|_\infty\leq (1+A)\frac{\mathbb H_i}{\lambda_{ii}}+\bar b_{ij}\|v\|_\infty+B.\]
\end{lemma}
\begin{proof}
We first show the existence of viscosity solutions of (\ref{u0}). By Proposition \ref{m3}, we tern to consider the following equation
\begin{equation}\label{F}
F_i(x,Du,u,v(x))=0,
\end{equation}
where $F_i(x,p,u_i,u_j):=H_i(x,-p,-u_i,u_j)$. 
By the definition of $F_i$, one can easily check that $\|F_i(x,0,0,0)\|_\infty=\|H_i(x,0,0,0)\|_\infty$,
\[F_i(x,p,u_i,u)-F_i(x,p,v_i,u)
   \geq \lambda_{ii}(u_i-v_i),\quad \forall u_i\geq v_i,\ (x,p,u)\in T^*M\times \mathbb R,\]
and for all $(x,v)\in M\times\mathbb R$, $u_i\neq v_i,u_j\neq v_j\in\mathbb R$,
\begin{equation*}
\bigg|\frac{F_i(x,0,0,u_j)-F_i(x,0,0,v_j)}{u_j-v_j}\cdot \frac{u_i-v_i}{F_i(x,0,u_i,v)-F_i(x,0,v_i,v)}\bigg|\leq b_{ij}.
\end{equation*}
By Lemma \ref{Iv}, $\frac{1}{\lambda_{ii}}\mathbb H_i+b_{ij}\|v\|_\infty$ (resp. $-\frac{1}{\lambda_{ii}}\mathbb H_i-b_{ij}\|v\|_\infty$) is a supersolution (resp. subsolution) of (\ref{F}).
By Perron's method, the viscosity solution $u_-$ of (\ref{F}) exists, which implies the existence of viscosity solutions of (\ref{u0}) by Proposition \ref{m3}. Moreover, we have $\|u_-\|_\infty\leq \frac{1}{\lambda_{ii}}\mathbb H_i+b_{ij}\|v\|_\infty$. By Proposition \ref{u+}, we conclude that
\begin{equation*}
\begin{aligned}
  \|u(x)\|_\infty&=\|-v_+\|_\infty
  \leq (1+\Theta\mu e^{\Theta\mu})\|u_-\|_\infty+C\mu e^{\Theta\mu}+\Theta\mu e^{\Theta\mu}\|v(x)\|_\infty
  \\ &\leq (1+A)(\frac{1}{\lambda_{ii}}\mathbb H_i+b_{ij}\|v\|_\infty)+B+A\|v(x)\|_\infty
  =(1+A)\frac{\mathbb H_i}{\lambda_{ii}}+\bar b_{ij}\|v\|_\infty+B,
\end{aligned}
\end{equation*}
where $v_+$ is a forward weak KAM solution of (\ref{F}).
\end{proof}

\begin{lemma}\label{Dunvn}
For $n=1,2,\dots$, we have
\begin{equation*}
  \|u^n_1\|_\infty \leq (1+A)\frac{\mathbb H_1}{\lambda_{11}}\sum_{l=0}^n\bar \kappa^l+\bar b_{12}(1+A)\frac{\mathbb H_2}{\lambda_{22}}\sum_{l=0}^{n-1} \bar \kappa^l+(\sum_{l=0}^n\bar \kappa^l+\bar b_{12}\sum_{l=0}^{n-1}\bar \kappa^l)B,
\end{equation*}
and
\begin{equation*}
  \|u^{n+1}_2\|_\infty \leq (1+A)\frac{\mathbb H_2}{\lambda_{22}}\sum_{l=0}^n\bar \kappa^l+\bar b_{21}(1+A)\frac{\mathbb H_1}{\lambda_{11}}\sum_{l=0}^n \bar \kappa^l+(\sum_{l=0}^n\bar \kappa^l+\bar b_{21}\sum_{l=0}^n\bar \kappa^l)B.
\end{equation*}
\end{lemma}
\begin{proof}
We prove by induction. We first prove the Lemma when $n=1$. By Lemma \ref{dv}, we have
\begin{equation*}
\begin{aligned}
  \|u^0_1(x)\|_\infty&\leq (1+A)\frac{\mathbb H_1}{\lambda_{11}}+B,
  \\ \|u^1_2(x)\|_\infty&\leq (1+A)\frac{\mathbb H_2}{\lambda_{22}}+\bar b_{21}\|u^0_1(x)\|_\infty+B
  \\ &\leq (1+A)\frac{\mathbb H_2}{\lambda_{22}}+\bar b_{21}(1+A)\frac{\mathbb H_1}{\lambda_{11}}+(1+\bar b_{21}) B,
  \\ \|u^1_1(x)\|_\infty&\leq (1+A)\frac{\mathbb H_1}{\lambda_{11}}+\bar b_{12}\|u^1_2(x)\|_\infty+B
  \\ &\leq (1+A)\frac{\mathbb H_1}{\lambda_{11}}(1+\bar \kappa)+\bar b_{12}(1+A)\frac{\mathbb H_2}{\lambda_{22}}+(1+\bar \kappa+\bar b_{12})B,
  \\ \|u^2_2(x)\|_\infty&\leq (1+A)\frac{\mathbb H_2}{\lambda_{22}}+\bar b_{21}\|u^1_1(x)\|_\infty+B
  \\ &\leq (1+A)\frac{\mathbb H_2}{\lambda_{22}}(1+\bar \kappa)+\bar b_{21}(1+A)\frac{\mathbb H_1}{\lambda_{11}}(1+\bar \kappa)+(1+\bar \kappa+\bar b_{21}(1+\bar \kappa))B.
\end{aligned}
\end{equation*}

Assume that the assertion holds true for $n=k-1$. We prove the Lemma when $n=k$. By Lemma \ref{dv}, we have
\begin{equation*}
\begin{aligned}
  \|u^k_1(x)\|_\infty&\leq (1+A)\frac{\mathbb H_1}{\lambda_{11}}+\bar b_{12}\|u^k_2(x)\|_\infty+B
  \\ &\leq (1+A)\frac{\mathbb H_1}{\lambda_{11}}(1+\bar b_{12}\bar b_{21}\sum_{l=0}^{k-1}\bar \kappa^l)+\bar b_{12}(1+A)\frac{\mathbb H_2}{\lambda_{22}}\sum_{l=0}^{k-1}\kappa^l
  \\ &\quad +\bar b_{12}(\sum_{l=0}^{k-1}\bar \kappa^l+\bar b_{21}\sum_{l=0}^{k-1}\bar \kappa^l)B+B
  \\ &=(1+A)\frac{\mathbb H_1}{\lambda_{11}}\sum_{l=0}^k\bar \kappa^l+\bar b_{12}(1+A)\frac{\mathbb H_2}{\lambda_{22}}\sum_{l=0}^{k-1}\bar\kappa^l+(\sum_{l=0}^k\bar \kappa^l+\bar b_{12}\sum_{l=0}^{k-1}\bar \kappa^l)B,
\end{aligned}
\end{equation*}
and
\begin{equation*}
\begin{aligned}
  \|u^{k+1}_2(x)\|_\infty&\leq (1+A)\frac{\mathbb H_2}{\lambda_{22}}+\bar b_{21}\|u^k_1(x)\|_\infty+B
  \\ &\leq (1+A)\frac{\mathbb H_2}{\lambda_{22}}(1+\bar b_{21}\bar b_{12}\sum_{l=0}^{k-1}\bar\kappa^l)+\bar b_{21}(1+A)\frac{\mathbb H_1}{\lambda_{11}}\sum_{l=0}^k\kappa^l
  \\ &\quad +\bar b_{21}(\sum_{l=0}^k\bar \kappa^l+\bar b_{12}\sum_{l=0}^{k-1}\bar \kappa^l)B+B
  \\ &=(1+A)\frac{\mathbb H_2}{\lambda_{22}}\sum_{l=0}^k\bar\kappa^l+\bar b_{21}(1+A)\frac{\mathbb H_1}{\lambda_{11}}\sum_{l=0}^{k}\bar\kappa^l+(\sum_{l=0}^k\bar \kappa^l+\bar b_{21}\sum_{l=0}^k\bar \kappa^l)B.
\end{aligned}
\end{equation*}
The proof is now completed.
\end{proof}
By assumption we have $\bar\kappa<1$, both $u^n_1$ and $u^n_2$ are uniformly bounded with respect to $n$. Similar to the proof of Proposition \ref{proI-I}, there exists a viscosity solution of (\ref{E}). The proof of Proposition \ref{proD-D} is now complete.

It remains to prove the existence of solutions of (\ref{E}) where (H8) is replaced by (H2), and (\ref{bark<1}) is replaced by $b_{12}b_{21}<1$. If $H_i(x,p,u_i,u_j)$ satisfies (H2) instead of (H8), then $L_i(x,\dot x,u_i,u_j)$ is finite for all $\dot x\in T_x M$. If $b_{12}b_{21}<1$, one can take $\mu>0$ as mentioned in Lemma \ref{mu} sufficiently small such that (\ref{bark<1}) holds. The proof of Theorem \ref{M1} when two equations in (\ref{E}) are both decreasing in the unknown function is now complete.

\begin{remark}\label{rem2}
It is natural to ask if we can prove the existence of solutions of (\ref{E}) when $\chi<1$ and (H8) holds instead of (H2) by modification. Let $K>0$, define
\[H^K_i(x,p,u_i,u_j)=H_i(x,p,u_i,u_j)+\max\{|p|_x^2-K^2,0\}.\]
Denote by $(u^K_1,u^K_2)$ the solution of (\ref{E}) with $H_i$ equaling $H^K_i$. We hope that $(u^K_1,u^K_2)$ uniformly converges as $K\to+\infty$. However, the constant $B$ in Lemma \ref{Dunvn} is not uniformly bounded with respect to $K$ for small $\mu>0$. We cannot prove that $(u^K_1,u^K_2)$ is uniformly bounded.
\end{remark}

\subsection{Increasing-decreasing case}

Without any loss of generality, we assume that $H_1(x,p,u_1,u_2)$ satisfies (I') and $H_2(x,p,u_2,u_1)$ satisfies (D') in this section.
\begin{proposition}\label{proI-D}
The system (\ref{E}) admits viscosity solutions if there is $\mu>0$ as mentioned in Lemma \ref{mu} such that
\begin{equation}\label{tilk<1}
  b_{12}\left((1+\Theta\mu e^{\Theta\mu})b_{21}+\Theta\mu e^{\Theta\mu}\right)<1.
\end{equation}
\end{proposition}

\begin{lemma}\label{IDunvn}
For $n=1,2,\dots$, we have
\begin{equation*}
  \|u^n_1\|_\infty \leq \frac{\mathbb H_1}{\lambda_{11}}\sum_{l=0}^n\tilde\kappa^l+b_{12}(1+A)\frac{\mathbb H_2}{\lambda_{22}}\sum_{l=0}^{n-1} \tilde\kappa^l+b_{12}B\sum_{l=0}^{n-1}\tilde\kappa^l.
\end{equation*}
and
\begin{equation*}
  \|u^{n+1}_2\|_\infty \leq (1+A)\frac{\mathbb H_2}{\lambda_{22}}\sum_{l=0}^n\tilde\kappa^l+\bar b_{21}\frac{\mathbb H_1}{\lambda_{11}}\sum_{l=0}^n \tilde\kappa^l+B\sum_{l=0}^n\tilde\kappa^l.
\end{equation*}
\end{lemma}
The proof of Lemma \ref{IDunvn} is quite similar to that of Lemmas \ref{unvn} and \ref{Dunvn}, we omit it here for brevity. By assumption we have $\tilde\kappa<1$, both $u^n_1$ and $u^n_2$ are uniformly bounded with respect to $n$. Similar to the proof of Proposition \ref{proI-I}, we can prove Proposition \ref{proI-D}.

It remains to prove the existence of solutions of (\ref{E}) where (H8) is replaced by (H2), and (\ref{tilk<1}) is replaced by $b_{12}b_{21}<1$. The proof is quite similar to that in Section \ref{secd-d}, based on the boundedness of $L_i(x,\dot x,u_i,u_j)$. 

\medskip

The proof of theorem \ref{M1} is now complete. 

\section{Proof of Theorem \ref{M3}}\label{Sec4}

The strategy of the proof of Theorem \ref{M3} is as follows. We use the vanishing discount method. Let $\varepsilon>0$, $c\in\mathbb R$ and $(u_1^\varepsilon,u_2^\varepsilon)$ be a viscosity solution of
\begin{equation}\label{Ee}
  \left\{
   \begin{aligned}
   &h_1(x,Du_1(x))+\Lambda_1(x)(u_1(x)-u_2(x))=c,\\
   &h_2(x,Du_2(x))+\varepsilon u_2(x)+\Lambda_2(x)(u_2(x)-u_1(x))=0.\\
   \end{aligned}
   \right.
\end{equation}
Proposition \ref{perronequ} and Lemma \ref{stuni} imply the existence and the uniqueness of the viscosity solution $(u_1^\varepsilon,u_2^\varepsilon)$ of (\ref{Ee}). We will prove that $\varepsilon u^\varepsilon_2(x)$ is uniformly bounded and $(u_1^\varepsilon,u_2^\varepsilon)$ is equi-Lipschitz continuous for all $\varepsilon>0$. Then for fixed $x_0\in M$, there exists a sequence $\varepsilon_k\rightarrow 0^+$ such that the pair $(u^{\varepsilon_k}_1(x)-u^{\varepsilon_k}_2(x_0),u^{\varepsilon_k}_2(x)-u^{\varepsilon_k}_2(x_0))$ uniformly converges to a viscosity solution of (\ref{CE}) and  $-\varepsilon_k u^{\varepsilon_k}_2(x_0)\rightarrow\alpha(c)\in\mathbb R$. 

\noindent \textbf{Step 1.} We first prove that for given $c\in\mathbb R$, there is a unique constant $\alpha(c)$ such that (\ref{CE}) admits viscosity solutions. Without any loss of generality, we may assume $c=0$ after a translation.
\begin{lemma}
The family $\{\varepsilon u^\varepsilon_2(x)\}_{\varepsilon>0}$ is uniformly bounded with respect to $\varepsilon$.
\end{lemma}
\begin{proof}
Let $x_1$ be a minimum point of $u^\varepsilon_1(x)$. Let $x_2$ be a minimum point of $u^\varepsilon_2(x)$. By Lemma \ref{equisub}, since $(u^\varepsilon_1,u^\varepsilon_2)$ is bounded for fixed $\varepsilon>0$, $u^\varepsilon_i$ is Lipschitz continuous. Then $u^\varepsilon_i$ is semiconcave for each $i\in\{1,2\}$ by \cite[Theorem 5.3.7]{sc}. Therefore, for each $i\in\{1,2\}$, $u^\varepsilon_i$ is differentiable at $x_i$, and $Du^\varepsilon_i(x_i)=0$. Plugging into (\ref{Ee}) we have
\begin{equation}\label{1}
   h_1(x_1,0)+\Lambda_1(x_1) (u^\varepsilon_1(x_2)-u^\varepsilon_2(x_2))
   \geq h_1(x_1,0)+\Lambda_1(x_1)( u^\varepsilon_1(x_1)-u^\varepsilon_2(x_1))=0,
\end{equation}
and
\begin{equation}\label{2}
  h_2(x_2,0)+\varepsilon u^\varepsilon_2(x_2)+\Lambda_2(x_2)(u^\varepsilon_2(x_2)-u^\varepsilon_1(x_2))=0.
\end{equation}
Multiplying (\ref{1}) with $\Lambda_2(x_2)/\Lambda_1(x_1)$ and adding with (\ref{2}), we get
\[\frac{\Lambda_2(x_2)}{\Lambda_1(x_1)}h_1(x_1,0)+h_2(x_2,0)+\varepsilon u^\varepsilon_2(x_2)\geq 0,\]
which implies
\[\varepsilon u^\varepsilon_2 (x)\geq \varepsilon u^\varepsilon_2 (x_2)\geq -(\iota\|h_1(x,0)\|_\infty+\|h_2(x,0)\|_\infty).\]
where $\iota:=\max_{x\in M}\Lambda_2(x)/\min_{x\in M}\Lambda_1(x)$.

Since both $u^\varepsilon_1$ and $u^\varepsilon_2$ are semiconcave, for all $\epsilon>0$ and each $i\in\{1,2\}$, there is a point $y_i$ such that $u^\varepsilon_i$ is differentiable at $y_i$, and
\begin{equation*}
  u^\varepsilon_i(y_i)\geq \max_{x\in M}u^\varepsilon_i(x)-\epsilon.
\end{equation*}
Plugging into (\ref{Ee}) we have
\begin{equation}\label{3}
   \begin{aligned}
   &h_1(y_1,Du^\varepsilon_1(y_1))+\Lambda_1(y_1)(u^\varepsilon_1(y_2)-u^\varepsilon_2(y_2))
   \\ &\leq h_1(y_1,Du^\varepsilon_1(y_1))+\Lambda_1(y_1) \max_{x\in M}u^\varepsilon_1(x)-\Lambda_1(y_1)(u^\varepsilon_2(y_1)-\epsilon)
   \\ &\leq h_1(y_1,Du^\varepsilon_1(y_1))+\Lambda_1(y_1)(u^\varepsilon_1(y_1)-u^\varepsilon_2(y_1))+2\Lambda_1(y_1) \epsilon=2\Lambda_1(y_1) \epsilon,
   \end{aligned}
\end{equation}
and
\begin{equation}\label{4}
  h_2(y_2,Du^\varepsilon_2(y_2))+\varepsilon u^\varepsilon_2(y_2)+\Lambda_2(y_2)(u^\varepsilon_2(y_2)-u^\varepsilon_1(y_2))=0.
\end{equation}
Multiplying (\ref{3}) with $\Lambda_2(y_2)/\Lambda_1(y_1)$ and adding with (\ref{4}), we get
\[\frac{\Lambda_2(y_2)}{\Lambda_1(y_1)}h_1(y_1,Du^\varepsilon_1(y_1))+h_2(y_2,Du^\varepsilon_2(y_2))+\varepsilon u^\varepsilon_2(y_2)\leq 2\Lambda_2(y_2) \epsilon,\]
which implies
\[\varepsilon u^\varepsilon_2(x)\leq \varepsilon \max_{x\in M} u^\varepsilon_2(x)\leq \iota\left|\min_{(x,p)\in T^*M}h_1(x,p)\right|+\left|\min_{(x,p)\in T^*M}h_2(x,p)\right|+(2\|\Lambda_2(x)\|_\infty+\varepsilon)\epsilon.\]
Let $\epsilon\rightarrow 0^+$, we get
\[\varepsilon u^\varepsilon_2(x)\leq \iota\left|\min_{(x,p)\in T^*M}h_1(x,p)\right|+\left|\min_{(x,p)\in T^*M}h_2(x,p)\right|.\]
By (h1), $\min_{(x,p)\in T^*M}h_i(x,p)$ is finite for each $i\in\{1,2\}$. Therefore, $\varepsilon u^\varepsilon_2(x)$ is uniformly bounded with respect to $\varepsilon$.
\end{proof}
\begin{lemma}
Both $u^\varepsilon_1$ and $u^\varepsilon_2$ are equi-Lipschitz continuous with respect to $\varepsilon$.
\end{lemma}
\begin{proof}
Multiplying the first equality of (\ref{Ee}) with $\Lambda_2(x)/\Lambda_1(x)$ and adding with the second equality of (\ref{Ee}), we get
\[\frac{\Lambda_2(x)}{\Lambda_1(x)}h_1(x,Du^\varepsilon_1(x))+h_2(x,Du^\varepsilon_2(x))=-\varepsilon u^\varepsilon_2(x),\]
which implies
\[h_1(x,Du^\varepsilon_1(x))\leq \tilde{\iota}\left(\iota\|h_1(x,0)\|_\infty+\|h_2(x,0)\|_\infty+\left|\min_{(x,p)\in T^*M}h_2(x,p)\right|\right),\]
and
\[h_2(x,Du^\varepsilon_2(x))\leq \iota\|h_1(x,0)\|_\infty+\|h_2(x,0)\|_\infty+\iota\left|\min_{(x,p)\in T^*M}h_1(x,p)\right|,\]
where $\tilde{\iota}:=\max_{x\in M}\Lambda_1(x)/\min_{x\in M}\Lambda_2(x)$.
Thus, both $u^\varepsilon_1$ and $u^\varepsilon_2$ are equi-Lipschitz continuous with respect to $\varepsilon$ according to Lemma \ref{equisub}.
\end{proof}

Fix $x_0\in M$ and define
\begin{equation*}
  \left\{
   \begin{aligned}
   &\tilde{u}^\varepsilon_1(x)=u^\varepsilon_1(x)-u^\varepsilon_2(x_0),\\
   &\tilde{u}^\varepsilon_2(x)=u^\varepsilon_2(x)-u^\varepsilon_2(x_0).\\
   \end{aligned}
   \right.
\end{equation*}
The pair $(\tilde{u}^\varepsilon_1,\tilde{u}^\varepsilon_2)$ satisfies
\begin{equation*}
  \left\{
   \begin{aligned}
   &h_1(x,D\tilde{u}^\varepsilon_1(x))+\Lambda_1(x)(\tilde{u}^\varepsilon_1(x)-\tilde{u}^\varepsilon_2(x))=0,\\
   &h_2(x,D\tilde{u}^\varepsilon_2(x))+\varepsilon\tilde{u}^\varepsilon_2(x)+\Lambda_2(x)(\tilde{u}^\varepsilon_1(x)-\tilde{u}^\varepsilon_2(x)) +\varepsilon u^\varepsilon_2(x_0)=0.\\
   \end{aligned}
   \right.
\end{equation*}
Since $\|Du^\varepsilon_1\|_\infty$ is uniformly bounded almost everywhere, according to the first equality of (\ref{Ee}), $\tilde{u}^\varepsilon_1(x_0)=u^\varepsilon_1(x_0)-u^\varepsilon_2(x_0)$ is uniformly bounded with respect to $\varepsilon$. We also have $\tilde{u}^\varepsilon_2(x_0)=0$. Thus, both $\tilde{u}^\varepsilon_1(x)$ and $\tilde{u}^\varepsilon_2(x)$ are uniformly bounded and equi-Lipschitz continuous with respect to $\varepsilon$. We have also proved that $-\varepsilon u^\varepsilon_2(x_0)$ is bounded. Therefore, there exists a sequence $\varepsilon_k\rightarrow 0^+$ such that the pair $(\tilde{u}^{\varepsilon_k}_1,\tilde{u}^{\varepsilon_k}_2)$ converges to $(u,v)$ uniformly, and $-\varepsilon_k u^{\varepsilon_k}_2(x_0)$ converges to a constant $d_0\in\mathbb R$. According to the stability of viscosity solutions, the limit $(u,v)$ is a viscosity solution of (\ref{CE}) with $c=0$ and $\alpha(c)=d_0$.

We now prove that $d_0$ is the unique constant such that (\ref{CE}) admits solutions. Here we still set $c=0$. Assume there are two different constants $d_1<d_2$ such that (\ref{CE}) has a solution $(u_1,u_2)$ (resp. $(v_1,v_2)$) with $\alpha(c)=d_1$ (resp. $\alpha(c)=d_2$). When $\varepsilon>0$ is small enough, we have $d_1+\varepsilon u_2\leq d_2+\varepsilon v_2$. Noticing that for $k\in\mathbb R$, the pair $(u_1+k,u_2+k)$ also solves (\ref{CE}) with $c=0$ and $\alpha(c)=d_1$. Let $k>0$ large enough, we may assume $(v_1,v_2)<(u_1,u_2)$. We get
\begin{equation*}
\left\{
\begin{aligned}
&h_1(x,Du_1)+\Lambda_1(x) (u_1-u_2)=0,
\\ &h_2(x,Du_2)+\varepsilon u_2+\Lambda_2(x)(u_2-u_1)=d_1+\varepsilon u_2\leq d_2+\varepsilon v_2.
\end{aligned}
\right.
\end{equation*}
According to Lemma \ref{stuni}, we have $(u_1,u_2)\leq (v_1,v_2)$, which contradicts $(v_1,v_2)<(u_1,u_2)$.

\noindent \textbf{Step 2.} We now prove the properties of the function $\alpha(c)$. Let $(u_1^{\varepsilon,c},u_2^{\varepsilon,c})$ be the viscosity solution of (\ref{Ee}). We take $c_1>c_2$. By Lemma \ref{stuni}, $u^{\varepsilon,c_1}_2\geq u^{\varepsilon,c_2}_2$. By Step 1, $\alpha(c)$ is the limit of a converging sequence $-\varepsilon_k u^{\varepsilon_k,c}_2(x)$. Therefore, we have $\alpha(c_1)\leq \alpha(c_2)$. Define
\[K_1:=\frac{\max_{x\in M}\Lambda_2(x)+\varepsilon}{\varepsilon\min_{x\in M}\Lambda_1(x)}(c_1-c_2),\quad K_2:=\frac{\max_{x\in M}\Lambda_2(x)}{\varepsilon\min_{x\in M}\Lambda_1(x)}(c_1-c_2).\]
Then we have
\begin{equation*}
  \left\{
   \begin{aligned}
   &h_1(x,Du^{\varepsilon,c_1}_1)+\Lambda_1(x)\bigl((u^{\varepsilon,c_1}_1-K_1)-(u^{\varepsilon,c_1}_2-K_2)\bigl)=c_1+\Lambda_1(x)(K_2-K_1)\leq c_2,\\
   &h_2(x,u^{\varepsilon,c_1}_2)+(\Lambda_2(x)+\varepsilon)(u^{\varepsilon,c_1}_2-K_2)-\Lambda_2(x)(u^{\varepsilon,c_1}_1-K_1)=\Lambda_2(x)(K_1-K_2)-\varepsilon K_2\leq 0.\\
   \end{aligned}
   \right.
\end{equation*}
which implies $u^{\varepsilon,c_1}_2-u^{\varepsilon,c_2}_2\leq K_2$ by Lemma \ref{stuni}. Consider the following identity
\[\alpha(c_1)-\alpha(c_2)=[\alpha(c_1)+\varepsilon u^{\varepsilon,c_1}_2]+[\varepsilon u^{\varepsilon,c_2}_2-\varepsilon u^{\varepsilon,c_1}_2]-[\varepsilon u^{\varepsilon,c_2}_2+\alpha(c_2)].\]
Since both $\varepsilon u^{\varepsilon,c_1}_2$ and $\varepsilon u^{\varepsilon,c_2}_2$ are sequentially compact, there is a sequence $\varepsilon_k\rightarrow 0^+$ such that $\varepsilon_k u^{\varepsilon_k,c_1}_2$ tends to $-\alpha(c_1)$ and $\varepsilon_k u^{\varepsilon_k,c_2}_2$ tends to $-\alpha(c_2)$. For each $\varepsilon>0$, we have \[\varepsilon u^{\varepsilon,c_2}_2-\varepsilon u^{\varepsilon,c_1}_2\geq -\varepsilon K_2=-\frac{\max_{x\in M}\Lambda_2(x)}{\min_{x\in M}\Lambda_1(x)}(c_1-c_2).\]
We conclude that
\[-\frac{\max_{x\in M}\Lambda_2(x)}{\min_{x\in M}\Lambda_1(x)}(c_1-c_2)\leq \alpha(c_1)-\alpha(c_2)\leq 0,\quad \forall c_1>c_2.\]

At last, we prove Corollary \ref{cons}. When $\Lambda_1(x)$ and $\Lambda_2(x)$ are constants, we have $u^{\varepsilon,c_1}_2-u^{\varepsilon,c_2}_2=K_2$ with $K_2=\frac{\Lambda_2}{\varepsilon\Lambda_1}(c_1-c_2)$. Therefore, the function $\alpha(c)$ is linear, with the slope $-\Lambda_2/\Lambda_1$.

\begin{remark}
In the proof of Theorem \ref{M3}, we showed that there is a sequence $\varepsilon_k\rightarrow 0^+$ such that $(\tilde{u}_1^{\varepsilon_k},\tilde{u}_2^{\varepsilon_k})$ converges. Without any loss of generality, we assume $c=\alpha(c)=0$ up to a translation. Let $(u_1^\varepsilon,u_2^\varepsilon)$ be the viscosity solution of (\ref{Ee}) with $c=0$. It is natural to ask whether $(u_1^\varepsilon,u_2^\varepsilon)$ converges to a viscosity solution of (\ref{CE}) uniformly as $\varepsilon\rightarrow 0^+$. The situation here is different from that in the previous works on the vanishing discount problem. The discounted term $\varepsilon u_2$ only appears in the second equation. Since we are not dealing with the vanishing discount problem in this paper, we will not discuss this problem here.
\end{remark}

\section*{Acknowledgements}

The author would like to thank Professor Jun Yan for many helpful discussions.

\textbf{Conflict of interest:} The author states no conflict of interest.

\appendix

\section{Single Hamilton-Jacobi equations}\label{Sec2}

Assume that $H_i:T^*M\times\mathbb R^2\to\mathbb R$ satisfies (H1)(H3)(H4)(H8) for each $i\in\{1,2\}$. Taking $v\in C(M)$. We collect some facts given by \cite{Ni} in view of the Hamiltonian defined by \[H(x,p,u):=H_i(x,p,u,v(x))\] for sake of completeness. In fact, if $H_i(x,p,u_i,u_j)$ satisfies (H1)(H3)(H4)(H8), then $H(x,p,u)$ satisfies the basic assumptions in \cite{Ni}, i.e. continuous in $(x,p,u)$, convex and coercive in $p$, and uniformly Lipschitz continuous in $u$. Correspondingly, one has the Lagrangian associated to $H$:
\begin{equation*}
  L(x,\dot x,u):=\sup_{p\in T^*_xM}\{\langle \dot x,p\rangle_x-H(x,p,u)\}.
\end{equation*}
\begin{remark}\label{C'}
The Lagrangian $L(x,\dot x,u)$ equals to $L_i(x,\dot x,u,v(x))$, where $L_i$ is defined in (\ref{Li}). One can check that $L_i(x,\dot x,u_i,u_j)$ is also uniformly Lipschitz continuous in $u_j$ with the Lipschitz constant $\Theta$. Let $\delta>0$ and $C>0$ be the constants defined in Lemma \ref{mu}, then
\[L(x,\dot x,0)=L_i(x,\dot x,0,v(x))\leq C+\Theta\|v(x)\|_\infty,\quad \forall (x,\dot x)\in M\times\bar B(0,\delta).\]
\end{remark}
\begin{proposition}\cite[Theorem 1]{Ni}\label{m1}
The backward solution semigroup
\begin{equation*}\label{T-}
  T^-_t\varphi(x)=\inf_{\gamma(t)=x} \left\{\varphi(\gamma(0))+\int_0^tL(\gamma(\tau),\dot{\gamma}(\tau),T^-_\tau\varphi(\gamma(\tau)))\textrm{d}\tau\right\}
\end{equation*}
is well-defined. The infimum is taken among absolutely continuous curves $\gamma:[0,t]\rightarrow M$ with $\gamma(t)=x$. If  $\varphi$ is  continuous, then $u(x,t):=T^-_t\varphi(x)$ represents the unique continuous viscosity solution of (\ref{ssC}). Define the forward solution semigroup by $T^+_t\varphi:=-\bar T^-_t(-\varphi)$, where $\bar T^-_t$ be the backward solution semigroup corresponding to $L(x,-\dot x,-u)$. 
\end{proposition}

Following Fathi \cite{Fat-b}, we introduce the forward weak KAM solutions. 
\begin{definition}\label{bws}
A function $u_+\in C(M)$ is called a forward weak KAM solution of (\ref{hjj}) if $-u_+$ is a viscosity solution of
\begin{equation}\label{FE}
H(x,-Du(x),-u(x))=0.
\end{equation}
\end{definition}

\begin{proposition}\cite[Theorem 3]{Ni}\label{m3}
Assume there is a viscosity solution $u_-$ of (\ref{hjj}), then $T^+_t u_-$ is nonincreasing in $t$, and converges to $u_+$ uniformly. Thus, the existence of viscosity solutions of (\ref{FE}) is equivalent to the existence of viscosity solutions of (\ref{hjj}). Moreover, the projected Aubry set with respect to $u_-$
\[\mathcal I_{u_-}:=\{x\in M:\ u_-=\lim_{t\to+\infty}T^+_tu_-\}\]
is nonempty.
\end{proposition}




In the following, we assume that $H_i(x,p,u,u_j)$ is strictly increasing in $u$, then so is $H(x,p,u)$. If there exists a viscosity solution $u_-$ of (\ref{hjj}), then by \cite[Theorem 3.2]{inc}, it is unique. 
\begin{proposition}\label{u+}
Let $v_+$ be a forward weak KAM solution of (\ref{hjj}), then the set
\[\mathcal I_{v_+}:=\{x\in M:\ v_+(x)=u_-(x)\}\]
is nonempty. Moreover, we have
\[u_-(y)-(C+\Theta\|v(x)\|_\infty+\Theta\|u_-\|_\infty)\mu e^{\Theta\mu}\leq v_+(x)\leq u_-(x),\quad y\in\mathcal I_{v_+}.\]
where $C>0$ and $\mu>0$ are constants given in Lemma \ref{mu}.
\end{proposition}
\begin{proof}
Since $u_-$ is unique, by Proposition \ref{m3}, $u_-=\lim_{t\to+\infty} T^-_tv_+$ for each forward weak KAM solution $v_+$, and $\mathcal I_{v_+}$ is nonempty.

Since $v_+\leq u_-$, we only need to show that $v_+$ has a lower bound. We take $y\in \mathcal I_{v_+}$ and $x\in M$, then $v_+(y)=u_-(y)$. Let $\alpha:[0,\mu]\rightarrow M$ be a geodesic satisfying $\alpha(0)=x$ and $\alpha(\mu)=y$ with constant speed, then $\|\dot \alpha\|\leq \delta$. If $v_+(x)\geq u_-(y)$, then the proof is finished. If $v_+(x)<u_-(y)$, since $v_+(y)=u_-(y)$, there is $\sigma\in(0,\mu]$ such that $v_+(\alpha(\sigma))=u_-(y)$ and $v_+(\alpha(s))<u_-(y)$ for all $s\in [0,\sigma)$. By \cite[Proposition 2.5]{Ish2}, we have
\begin{equation*}
\begin{aligned}
  v_+(\alpha(\sigma))-v_+(\alpha(s))\leq \int_s^\sigma L(\alpha(\tau),\dot \alpha(\tau),v_+(\alpha(\tau)))d\tau,
\end{aligned}
\end{equation*}
which implies
\begin{equation*}
\begin{aligned}
  &u_-(y)-v_+(\alpha(s))
  \leq \int_s^\sigma L(\alpha(\tau),\dot \alpha(\tau),v_+(\alpha(\tau)))d\tau
  \\ &\leq \int_s^\sigma L(\alpha(\tau),\dot \alpha(\tau),u_-(y))d\tau
  +\Theta\int_s^\sigma(u_-(y)-v_+(\alpha(\tau)))d\tau
  \\ &\leq L_0\mu+\Theta\int_s^\sigma(u_-(y)-v_+(\alpha(\tau)))d\tau,
\end{aligned}
\end{equation*}
where
\begin{equation*}
  L_0:=C+\Theta\|v(x)\|_\infty+\Theta\|u_-\|_\infty.
\end{equation*}
Let $G(\sigma-s)=u_-(y)-v_+(\alpha(s))$, then
\begin{equation*}
  G(\sigma-s)\leq L_0\mu+\Theta\int_0^{\sigma-s}G(\tau)d\tau,
\end{equation*}
By the Gronwall inequality we get
\begin{equation*}
  u_-(y)-v_+(\alpha(s))\leq L_0\mu e^{\Theta(\sigma-s)}\leq L_0\mu e^{\Theta\mu},\quad \forall s\in[0,\sigma).
\end{equation*}
Therefore $v_+(x)\geq u_-(y)-L_0\mu e^{\Theta\mu}$.
\end{proof}

\section{Proof of Theorem \ref{M4}}\label{sec3}

Let $(u_1,u_2)$ be the solution of (\ref{LCau}). 
We take $r>b_{12}>0$ to be close to $b_{12}$. Define $v_1=u_1/r$, then the pair $(v_1,u_2)$ satisfies
\begin{equation*}
  \left\{
   \begin{aligned}
   &\partial_t v_1+r^{-1}h_1(x,rDv_1)+\lambda_{11}(x)v_1+r^{-1}\lambda_{12}(x)u_2=0,\\
   &\partial_t u_2+h_2(x,Du_2)+\lambda_{22}(x)u_2+r\lambda_{21}(x)v_1=0.\\
   \end{aligned}
   \right.
\end{equation*}
Similar to Proposition \ref{perronequ}, we get $\lambda_{11}(x)+r^{-1}\lambda_{12}(x)>0$ 
and $\lambda_{22}(x)+r\lambda_{21}(x)>0$. In all the following proofs, we always assume (\ref{perron}), since $u_1=rv_1$, $r>0$ and $(v_1,u_2)$ satisfies (\ref{LE}) with (\ref{perron}) holds. If $(v_1,u_2)$ converges as $t\to+\infty$, then $(u_1,u_2)$ also converges. The proofs of Lemmas \ref{comp1} and \ref{unibdd} are standard, for similar results, one can refer to \cite[Appendix A.1]{ran}. It is worth mentioning that Lemmas \ref{comp1} and \ref{unibdd} also hold when $\chi\leq 1$. 

\begin{lemma}\label{comp1}
Assume (h1) (\ref{mono2}) and $\chi<1$. For each $T>0$, let $(u_1,u_2)$ (resp. $(v_1,v_2)$) be a bounded u.s.c. subsolution (resp. bounded l.s.c. supersolution) of
\begin{equation}\label{evol}
  \partial_t u_i+h_i(x,Du_i)+\sum_{j=1}^2\lambda_{ij}(x)u_j=0,\quad i\in\{1,2\},\quad (x,t)\in M\times(0,T)
\end{equation}
with $(u_1(x,0),u_2(x,0))\leq (\varphi_1,\varphi_2)$ (resp. $(v_1(x,0),v_2(x,0))\geq (\varphi_1,\varphi_2)$). If either $(u_1,u_2)$ or $(v_1,v_2)$ is Lipschitz continuous, then $u_i\leq v_i$ for $i\in\{1,2\}$.
\end{lemma}
\begin{proof}
The proof is a standard doubling variable method. We argue by contradiction. Assume $\max_{i\in\{1,2\}}\max_{(x,t)\in M\times[0,T]}(u_i(x,t)-v_i(x,t))>0$. Then there is a small $\eta>0$ such that
\[\bar M:=\max_{i\in\{1,2\}}\max_{(x,t)\in M\times[0,T]}(u_i(x,t)-\eta t-v_i(x,t))>0.\]
This maximum can be attained at $(x_0,t_0,i_0)$ with $t_0\in (0,T]$, since $M\times[0,T]$ is compact, and $u_i(x,0)\leq v_i(x,0)$ for $i=1,2$. Choosing a chart around $x_0$, we can assume $x_0=0$ and that $M$ is an open bounded subset $U$ of $\mathbb R^n$. Define
\[\Psi(x,y,t,s,j)=u_j(x,t)-v_j(y,s)-\frac{|x-y|^2}{2\alpha}-\frac{|t-s|^2}{2\mu}-|x|^2-\eta t,\]
where $\alpha$ and $\mu$ are positive constants. In the following, we are going to take $\alpha,\mu\to 0$. The function $\Psi(x,y,s,t,j)$ is upper-semicontinuous, then the maximum of $\Psi$ can be achieved at $(\bar x,\bar y,\bar s, \bar t,\bar j)\in \bar U\times\bar U\times [0,T]\times [0,T]\times\{1,2\}$. By
\begin{equation}\label{psi>}
  \Psi(\bar x,\bar y,\bar t,\bar s,\bar j)\geq \Psi(0,0,t_0,t_0,i_0)=\bar M,
\end{equation}
we get that
  $\frac{|\bar x-\bar y|^2}{2\alpha}$ and $\frac{|\bar t-\bar s|^2}{2\mu}$
are bounded, which implies $|\bar x-\bar y|\to 0$ and $|\bar t-\bar s|\to 0$ as $\alpha \to 0 $ and $\mu\to 0$. We now show that $\bar x\to 0$ and $\bar y\to 0$. Otherwise, there are sequences $\alpha_n\to 0$ and $\mu_n\to 0$ such that $\bar x\to z$, $\bar y\to z$, $z\neq 0$, $\bar t\to\tau$ and $\bar s\to \tau$. By the semi-continuity of $(u_1,u_2)$ and $(v_1,v_2)$, for $\alpha_n$ and $\mu_n$ small enough, we get
\[u_{\bar j}(\bar x,\bar t)-v_{\bar j}(\bar y,\bar s)-\frac{|\bar x-\bar y|^2}{2\alpha}-\frac{|\bar t-\bar s|^2}{2\mu}-|\bar x|^2-\eta \bar t\leq u_{\bar j}(z,\tau)-v_{\bar j}(z,\tau)-|z|^2-\eta \tau+\varepsilon<\bar M,\]
where $\varepsilon>0$ is small. This contradicts (\ref{psi>}). Thus, $\bar x$ and $\bar y$ are contained in the open set $U$ for $\alpha$ and $\mu$ small enough. 

Since $(u_1,u_2)$ is a subsolution, we get
\[\frac{\bar t-\bar s}{2\mu}+h_{\bar j}(\bar x,\bar p+2\bar x)+\sum_{k=1}^2\lambda_{\bar jk}(\bar x)u_k(\bar x,\bar t)\leq -\eta.\]
Since $(v_1,v_2)$ is a supersolution, we get
\[\frac{\bar t-\bar s}{2\mu}+h_{\bar j}(\bar y,\bar p)+\sum_{k=1}^2\lambda_{\bar jk}(\bar y)v_k(\bar y,\bar s)\geq 0.\]
Here we set $\bar p=\frac{\bar x-\bar y}{\alpha}$. Subtracting the above two inequalities, we get
\begin{equation}\label{eta<}
  h_{\bar j}(\bar x,\bar p+2\bar x)-h_{\bar j}(\bar y,\bar p)+\sum_{k=1}^2(\lambda_{\bar jk}(\bar x)u_k(\bar x,\bar t)-\lambda_{\bar jk}(\bar y)v_k(\bar y,\bar s))\leq -\eta.
\end{equation}
Since $u_{\bar j}(\bar x,\bar t)-v_{\bar j}(\bar y,\bar s)\geq \Psi(\bar x,\bar y,\bar t,\bar s,\bar j)>0$, we get \[\sum_{k=1}^2\lambda_{\bar jk}(\bar x)u_k(\bar x,\bar t)\geq \sum_{k=1}^2\lambda_{\bar jk}(\bar x)v_k(\bar y,\bar s).\]
Note that $\bar x\to 0$, $\bar y \to 0$, and $(u_1,u_2)$ is bounded. If $\bar p$ is bounded, we can let $\alpha\to 0$ and $\mu\to 0$ in (\ref{eta<}) to get $0\leq -\eta$. This leads to a contradiction.

It remains to prove that $\bar p$ is bounded. We take $x=y=\bar x$, $t=\bar t$ and $s=\bar s$ in $\Psi$, then
\[u_j(\bar x,\bar t)-v_j(\bar x,\bar s)-\frac{|\bar t-\bar s|^2}{2\mu}-|\bar x|^2-\eta \bar t
\leq u_j(\bar x,\bar t)-v_j(\bar y,\bar s)-\frac{|\bar x-\bar y|^2}{2\alpha}-\frac{|\bar t-\bar s|^2}{2\mu}-|\bar x|^2-\eta \bar t,\]
which implies
\[\frac{|\bar x-\bar y|^2}{2\alpha}\leq v_j(\bar x,\bar s)-v_j(\bar y,\bar s).\]
Now we assume that $(v_1,v_2)$ is Lipschitz continuous. Let $\kappa$ be the Lipschitz constant of $(v_1,v_2)$. Then
\[\frac{|\bar x-\bar y|^2}{2\alpha}\leq \kappa|\bar x-\bar y|.\]
We conclude that
\[|\bar p|=\frac{|\bar x-\bar y|}{\alpha}\leq 2\kappa.\]
The proof is now complete.
\end{proof}


\begin{lemma}\label{unibdd}
Assume (h1) and $\chi<1$. For all $(\varphi_1,\varphi_2)\in C(M,\mathbb R^2)$, the viscosity solution $(u_1,u_2)$ of (\ref{LCau}) exists. It is unique and uniformly bounded for all $t>0$. If the initial function $(\varphi_1,\varphi_2)$ is Lipschitz continuous, $(u_1,u_2)$ is equi-Lipschitz continuous for $(x,t)\in M\times[0,+\infty)$.
\end{lemma}
\begin{proof}
We first assume that $(\varphi_1,\varphi_2)$ is of class $C^1$. We first show the existence of $(u_1,u_2)$ by Perron's method.  Define
\[K_3:=\max_{i\in\{1,2\}}\{|h_i(x,p)|: \ x\in M,\ |p|\leq \|D\varphi_i\|_\infty\},\]
and
\[K_4:=K_3+\max_{i\in\{1,2\}}\sum_{j=1}^2\lambda_{ij}(x)\|\varphi_j(x)\|_\infty.\]
By (\ref{perron}), $(\varphi_1+K_4 t,\varphi_2+K_4 t)$ (resp. $(\varphi_1-K_4t,\varphi_2-K_4 t)$) is a supersolution (resp. subsolution) of (\ref{evol}). For each $T>0$, we define a family of subsolutions
\begin{equation*}
\begin{aligned}
  \mathcal S:=\{&(w_1,w_2):
  \\ &(w_1,w_2)\ \textrm{is}\ \textrm{a}\ \textrm{subsolution}\ \textrm{of}\ (\ref{evol})\ \textrm{with}\ (w_1,w_2)\leq (\varphi_1+K_4t,\varphi_2+K_4t)\}.
\end{aligned}
\end{equation*}
By \cite[Theorem 3.3]{Ish3}, the pair $(u_1,u_2)$ with component
\[u_i(x,t):=\sup_{(w_1,w_2)\in\mathcal S} w_i(x,t)\]
is a solution of (\ref{LCau}).

By Proposition \ref{perronequ} and Lemma \ref{stuni}, (\ref{LE}) has a unique solution when $\chi<1$. Let $(v_1,v_2)$ be the solution of $(\ref{LE})$. By (h1) and Lemma \ref{equisub}, $(v_1,v_2)$ is Lipschitz continuous. By (\ref{perron}), we can take $K_5>0$ large enough such that $(v_1+K_5,v_2+K_5)$ and $(v_1-K_5,v_2-K_5)$ are supersolution and subsolution of (\ref{evol}) respectively. Then by Lemma \ref{comp1}, we have $(v_1-K_5,v_2-K_5)\leq (u_1,u_2)\leq (v_1+K_5,v_2+K_5)$, which implies that $(u_1,u_2)$ is uniformly bounded.

Now we prove the regularity of $(u_1,u_2)$. For each $h>0$, we define
\begin{equation*}
  \bar w_i(x,t)=\left\{
   \begin{aligned}
   &\varphi_i(x)-K_4t,\quad 0\leq t\leq h\\
   &u_i(x,t-h)-K_4h,\quad t>h.\\
   \end{aligned}
   \right.
\end{equation*}
For $t<h$, $(\bar w_1,\bar w_2)$ is a classical subsolution. Since $(\varphi_1-K_4 t,\varphi_2-K_4 t)\leq (u_1,u_2)\leq (\varphi_1+K_4t,\varphi_2+K_4 t)$, there is no smooth function supertangent to $(\bar w_1,\bar w_2)$ at $t=h$ unless $(\bar w_1,\bar w_2)$ is differentiable at $t=h$. Thus, $(\bar w_1,\bar w_2)$ is a subsolution for $t\leq h$. For $t>h$, by (\ref{perron}) we have
\begin{equation*}
   \begin{aligned}
   &\partial_t \bar w_i+h_i(x,D\bar w_i)+\sum_{j=1}^2\lambda_{ij}(x)\bar w_j(x)
   \\ &=\partial_t u_i(x,t-h)+h_i(x,Du_i(x,t-h))+\sum_{j=1}^2\lambda_{ij}(x)(u_j(x,t-h)-K_4h)
   \\ &\leq \partial_t u_i(x,t-h)+h_i(x,Du_i(x,t-h))+\sum_{j=1}^2\lambda_{ij}(x)u_j(x,t-h)=0
   \end{aligned}
\end{equation*}
in the viscosity sense. Thus, $(\bar w_1,\bar w_2)$ is a subsolution of (\ref{evol}) with $\bar w_i(x,0)=\varphi_i(x)$. Since $(u_1,u_2)$ is the supremum of $\mathcal S$, we have
\[\bar w_i(x,t+h)=u_i(x,t)-K_4h\leq u_i(x,t+h).\]
Therefore, we have $\partial_t u_i\geq -K_4$. By (\ref{evol}), (h1), Lemma \ref{equisub} and the uniform boundedness of $(u_1,u_2)$, we obtain that $\|Du_i\|_\infty$ is uniformly bounded almost everywhere. Then $\|\partial_t u_i\|_\infty$ is uniformly bounded almost everywhere.

The uniqueness of the solution of (\ref{LCau}) when the initial function is of class $C^1$ is given by Lemma \ref{comp1}. Now we consider the existence and uniqueness of (\ref{LCau}) when $(\varphi_1,\varphi_2)$ is continuous. Let $(\varphi_1,\varphi_2)$ be continuous. Then there is a sequence of smooth functions $(\varphi^n_1,\varphi^n_2)$ uniformly converges to $(\varphi_1,\varphi_2)$. We have known that the solution $(u^n_1,u^n_2)$ of (\ref{evol}) with initial function equals $(\varphi^n_1,\varphi^n_2)$ is unique and Lipschitz continuous. For $n_1$ and $n_2\in\mathbb N$, define
\[K_6:=\max_{i}\|\varphi^{n_1}_i-\varphi^{n_2}_i\|_\infty.\]
The pair $(u^{n_1}_1-K_6,u^{n_1}_2-K_6)$ is a subsolution of (\ref{evol}) with initial function smaller than $(\varphi^{n_2}_1,\varphi^{n_2}_2)$. By Lemma \ref{comp1}, for each $i\in\{1,2\}$, we have
\[u^{n_1}_i-K_6\leq u^{n_2}_i.\]
Exchanging $n_1$ and $n_2$, we get
\[\|u^{n_1}_i-u^{n_2}_i\|_\infty \leq K_6.\]
Therefore, $(u^n_1,u^n_2)$ is a Cauchy sequence, and uniformly converges to a continuous function $(u_1,u_2)$. By the stability of viscosity solutions, $(u_1,u_2)$ is a solution of (\ref{LCau}). Define
\[K_7:=\max_{i}\|\varphi_i-\varphi^n_i\|_\infty.\]
By the Lipschitz continuity of $(u^n_1,u^n_2)$, for all solutions $(u_1,u_2)$ of (\ref{LCau}), we have
\[u^n_i-K_7\leq u_i,\]
and
\[u_i-K_7\leq u^n_i,\]
which implies
\[\|u_i-u^n_i\|_\infty \leq K_7.\]
Thus, the solution of (\ref{LCau}) is unique. When $(\varphi_1,\varphi_2)$ is Lipschitz continuous, we can take an equi-Lipschitz continuous sequence $(\varphi^n_1,\varphi^n_2)\to (\varphi_1,\varphi_2)$. Then $(u^n_1,u^n_2)$ is equi-Lipschitz continuous by the argument above. We then get the Lipschitz continuity of $(u_1,u_2)$.
\end{proof}

Now we consider the large time behavior of solution of (\ref{LCau}). Assume that $(\varphi_1,\varphi_2)$ is Lipschitz continuous. For each $i\in\{1,2\}$, define
\[\check{\varphi}_i(x)=\liminf_{t\to+\infty}u_i(x,t),\]
and
\[V_i(x,t)=\inf_{s\geq 0}u_i(x,t+s).\]
By \cite[Lemma 3.1]{Ish2}, the pair $(V_1(x,t),V_2(x,t))$ is a supersolution of (\ref{evol}). By definition, $(V_1(x,t),V_2(x,t))$ is nondecreasing in $t$. Then the pointwise limit $\check{\varphi}_i(x):=\lim_{t\to+\infty}V_i(x,t)$ exists. By the equi-Lipschitz continuity of $(u_1,u_2)$, $(\check{\varphi}_1,\check{\varphi}_2)$ is continuous. By Dini's theorem, the limit procedure is uniform. By the stability of viscosity solutions, $(\check{\varphi}_1,\check{\varphi}_2)$ is a supersolution of (\ref{LE}). Similarly, define
\[\hat{\varphi}_i(x)=\limsup_{t\to+\infty}u_i(x,t),\]
then $(\hat{\varphi}_1,\hat{\varphi}_2)$ is a subsolution of (\ref{E}). By definition, $(\check{\varphi}_1,\check{\varphi}_2)\leq (\hat{\varphi}_1,\hat{\varphi}_2)$. By Lemma \ref{stuni}, $(\hat{\varphi}_1,\hat{\varphi}_2)\leq (v_1,v_2)\leq (\check{\varphi}_1,\check{\varphi}_2)$, which implies
\[\lim_{t\to+\infty}u_i(x,t)=v_i(x)\]
uniformly for each $i\in\{1,2\}$.

For continuous initial function, let $(u^-_1,u^-_2)$ (resp. $u^+_1,u^+_2$) be the unique solution of (\ref{evol}) with constant initial function $(-\|\varphi_1\|_\infty,-\|\varphi_2\|_\infty)$ (resp. $(\|\varphi_1\|_\infty,\|\varphi_2\|_\infty)$). Then $(u^-_1,u^-_2)\leq (u_1,u_2)\leq (u^+_1,u^+_2)$. We have shown that
\[\lim_{t\to+\infty}u^\pm_i(x,t)=v_i(x)\]
for each $i$, then
\[\lim_{t\to+\infty}u_i(x,t)=v_i(x)\]
uniformly for each $i\in\{1,2\}$.

The proof is now complete.




\medskip

\begin{thebibliography}{}\label{sec:TeXbooks}

\bibitem{ABR} A. Arapostathis, A. Biswas and P. Roychowdhury. \emph{On ergodic control problem for viscous Hamilton-Jacobi equations for weakly coupled elliptic systems}. J. Diff. Equ., \textbf{314} (2022), 128--160.


\bibitem{CGT} F. Cagnetti, D. Gomes, and H. V. Tran. \emph{Adjoint methods for obstacle problems and weakly coupled systems of PDE}. ESAIM Control Optim. Calc. Var., \textbf{19} (2013), 754--779.

\bibitem{vis1} F. Cagnetti, D. Gomes, H. Mitake, H. V. Tran. \emph{A new method for large time behavior of degenerate viscous Hamilton-Jacobi equations with convex Hamiltonians}. Ann. Inst. H. Poincar\'e Anal. Non Lin\'eaire \textbf{32} (2015), 183--200.

\bibitem{Cam1} F. Camilli, O. Ley and P. Loreti. \emph{Homogenization of monotone systems of Hamilton-Jacobi equations}. ESAIM Control Optim. Calc. Var., \textbf{16} (2010), 58--76.

\bibitem{Cam2} F. Camilli, O. Ley, P. Loreti and V. Duc Nguyen. \emph{Large time behavior of weakly coupled systems of first-order Hamilton-Jacobi equations}. Nonlinear Differ. Equ. Appl., \textbf{19} (2012), 719--749.



\bibitem{sc} P. Cannarsa, C. Sinestrari. \emph{Semiconcave functions, Hamilton-Jacobi equations, and optimal control}. vol. 58. Springer, New York, 2004.

\bibitem{inc} Q. Chen, W. Cheng, H. Ishii and K. Zhao. \emph{Vanishing contact structure problem and convergence of the viscosity solutions}. Comm. Partial Differential Equations, \textbf{44} (2019), 801--836.




\bibitem{Davi1} A. Davini and M. Zavidovique. \emph{Aubry sets for weakly coupled systems of Hamilton-Jacobi equations}. SIAM J. Math. Anal., \textbf{46} (2014), 3361--3389.

\bibitem{ran} A. Davini, A. Siconolfi and M. Zavidovique. \emph{Random Lax-Oleinik semigroups for Hamilton-Jacobi systems}. J. Math. Pures Appl., \textbf{120} (2018), 294--333.

\bibitem{Davi2} A. Davini and M. Zavidovique. \emph{Convergence of the solutions of discounted Hamilton-Jacobi systems}. Advances in Calculus of Variations, \textbf{14} (2019), 1--15.


\bibitem{Eng} H. Engler and S. M. Lenhart. \emph{Viscosity solutions for weakly coupled systems of Hamilton-Jacobi equations}. Proc. London Math. Soc., \textbf{63} (1991), 212--240.

\bibitem{Fat-b} A. Fathi, \emph{Weak KAM Theorem in Lagrangian Dynamics}, preliminary version 10, Lyon, unpublished, 2008.


\bibitem{opt1} W. H. Fleming and Q. Zhang. \emph{Risk-sensitive production planning of a stochastic manufacturing system}. SIAM J. Control Optim., \textbf{36} (1988), 1147--1170.


\bibitem{Ish} H. Ishii. \emph{Perron's methods for Hamilton-Jacobi equations}. Duke Math. J., \textbf{55} (1987), 369--384.

\bibitem{Ish2} H. Ishii. \emph{Asymptotic solutions for large time of Hamilton-Jacobi equations in Euclidean n space}. Ann. Inst. H. Poincar\'e Anal. Non Lin\'eaire, \textbf{25} (2008), 231--266.

\bibitem{Ish3} H. Ishii and S. Koike. \emph{Viscosity solutions for monotone systems of second-order elliptic PDEs}. Comm. Partial Differential Equations, \textbf{16} (1991), 1095--1128.



\bibitem{Ish5} H. Ishii and L. Jin. \emph{The vanishing discount problem for monotone systems of Hamilton-Jacobi equations. Part 2: nonlinear coupling}. Calc. Var., \textbf{59} (2020), 1--28.

\bibitem{Ish6} H. Ishii. \emph{A short introduction to viscosity solutions and the large time behavior of solutions of Hamilton-Jacobi equations}. Hamilton-Jacobi Equations: Applications, Numerical Analysis and Applications, Lecture Notes in Math, Vol. 2074, Heidelberg: Springer, 111--249, 2013.


\bibitem{JL2} L. Jin, L. Wang and J. Yan. \emph{A representation formula of viscosity solutions to weakly coupled systems of Hamilton-Jacobi equations with applications to regularizing effect}. J. Diff. Equ., \textbf{268} (2020), 2012--2039.




\bibitem{Len} S. M. Lenhart. \emph{Viscosity solutions for weakly coupled systems of first-order partial differential equations}. J. Math. Anal. Appl., \textbf{131} (1988), 180--193.



\bibitem{Mit2} H. Mitake and H. V. Tran. \emph{Remarks on the large-time behavior of viscosity solutions of quasi-monotone weakly coupled systems of Hamilton-Jacobi equations}.  Asymptotic Analysis, \textbf{77} (2012), 43--70.

\bibitem{Mit4} H. Mitake and H. V. Tran. \emph{Homogenization of weakly coupled systems of Hamilton-Jacobi equations with fast switching rates}. Arch. Rational Mech. Anal. \textbf{211} (2014), 733--769.


\bibitem{Ng} V. D. Nguyen. \emph{Some results on the large time behavior of weakly coupled systems of first-order Hamilton-Jacobi equations}. Journal of Evolution Equations, \textbf{14} (2014), 299--331.

\bibitem{Ni} P. Ni, L. Wang and J. Yan, \emph{A representation formula of the viscosity solution of the contact Hamilton-Jacobi equation and its applications}. Chinese Annals of Mathematics, Series B, in press. arXiv: 2101.00446.

\bibitem{Ni2} P. Ni, K. Wang and J. Yan, \emph{A weakly coupled mean field games model of first order for $k$ groups of major players}. Proc. Amer. Math. Soc., published online.


\bibitem{SZ} A. Siconolfi and S. Zabad. \emph{Scalar reduction techniques for weakly coupled Hamilton-Jacobi systems}. Nonlinear Differ. Equ. Appl., \textbf{25} (2018), 1--20.

\bibitem{Su} X. Su, L. Wang and J. Yan. \emph{Weak KAM theory for Hamilton-Jacobi equations depending on unknown functions}. Discrete Contin. Dyn. Syst. \textbf{36} (2016), 6487--6522.




\bibitem{Wa5} K. Wang, J. Yan and K. Zhao. \emph{Time periodic solutions of Hamilton-Jacobi equations with autonomous Hamiltonian on the circle}. J. Math. Pures Appl., \textbf{171} (2023), 122--141.

\bibitem{opt2} G. G. Yin and Q. Zhang. \emph{Continuous-time Markov chains and applications}, volume 37 of Applications of Mathematics (New York). Springer-Verlag, New York, 1998. A singular perturbation approach.

\bibitem{Z} M. Zavidovique. \emph{Twisted Lax-Oleinik formulas and weakly coupled systems of Hamilton-Jacobi equations}. Annales de la facult\'e des sciences de Toulouse, \textbf{2} (2019), 209--224.




\end{thebibliography}
\end{document}